\documentclass{amsart}

\usepackage{amsmath}
\usepackage{hyperref}
\usepackage{amsfonts,graphics,amsthm,amsfonts,amscd,latexsym, amssymb}

\DeclareFontFamily{U}{MnSymbolC}{}
\DeclareSymbolFont{MnSyC}{U}{MnSymbolC}{m}{n}
\DeclareFontShape{U}{MnSymbolC}{m}{n}{
    <-6>  MnSymbolC5
   <6-7>  MnSymbolC6
   <7-8>  MnSymbolC7
   <8-9>  MnSymbolC8
   <9-10> MnSymbolC9
  <10-12> MnSymbolC10
  <12->   MnSymbolC12}{}
\DeclareMathSymbol{\iprod}{\mathbin}{MnSyC}{'270}

\usepackage{epsfig}
\usepackage{flafter}
\usepackage{mathtools}
\usepackage{comment}
\usepackage{stmaryrd}

\usepackage{mathabx,epsfig}

\hypersetup{
    colorlinks=true,    
    linkcolor=blue,          
    citecolor=blue,      
    filecolor=blue,      
    urlcolor=blue           
}

\usepackage{pgfplots}
\pgfplotsset{compat = newest}

\usepackage{tikz}
\usetikzlibrary{graphs,positioning,arrows,shapes.misc,decorations.pathmorphing}

\tikzset{
    >=stealth,
    every picture/.style={thick},
    graphs/every graph/.style={empty nodes},
}

\tikzstyle{vertex}=[
    draw,
    circle,
    fill=black,
    inner sep=1pt,
    minimum width=5pt,
]
\usepackage[position=top]{subfig}
\usepackage{amssymb}
\usepackage{color}

\calclayout

\usetikzlibrary{decorations.pathmorphing}
\tikzstyle{printersafe}=[decoration={snake,amplitude=0pt}]

\newcommand{\grad}{\operatorname{grad}}

\newcommand{\del}{\partial}
\newcommand{\divv}{\operatorname{div}}

\newcommand{\QQ}{\mathbb{Q}}

\newcommand{\RR}{\mathbb{R}}

\newcommand{\ReMer}{\mathcal{M}}
\newcommand{\Alg}{\mathcal{A}}

\newcommand{\godd}{\mathfrak{g}}
\newcommand{\vodd}{\mathfrak{v}}

\def\O#1.{\mathcal {O}_{#1}}			
\def\pr #1.{\mathbb P^{#1}}				
\def\af #1.{\mathbb A^{#1}}			
\def\ses#1.#2.#3.{0\to #1\to #2\to #3 \to 0}	
\def\xrar#1.{\xrightarrow{#1}}			
\def\K#1.{K_{#1}}						
\def\bA#1.{\mathbf{A}_{#1}}			
\def\bM#1.{\mathbf{M}_{#1}}				
\def\bL#1.{\mathbf{L}_{#1}}				
\def\bB#1.{\mathbf{B}_{#1}}				
\def\bK#1.{\mathbf{K}_{#1}}			
\def\subs#1.{_{#1}}					
\def\sups#1.{^{#1}}

\usepackage{tikz}
\usetikzlibrary{matrix,arrows,decorations.pathmorphing}

\newtheorem{introdef}{Definition}

  \newtheorem{introthm}{Theorem}

  \newtheorem{theorem}{Theorem}[section]
  \newtheorem{lemma}[theorem]{Lemma}
  \newtheorem{proposition}[theorem]{Proposition}

  \newtheorem{conjecture}[theorem]{Conjecture}

  \newtheorem{definition}[theorem]{Definition}
  \newtheorem{example}[theorem]{Example}

\newtheorem{remark}[theorem]{Remark}

\theoremstyle{remark}

\numberwithin{equation}{section}

\usepackage[all]{xy}

\begin{document}

\title{Odd Metrics}

\author[L.~Braun]{Lukas Braun}
\address{Mathematisches Institut, Albert-Ludwigs-Universit\"at Freiburg, Ernst-Zermelo-Strasse 1, 79104 Freiburg im Breisgau, Germany}
\email{lukas.braun@math.uni-freiburg.de}

\thanks{
LB is supported by the Deutsche Forschungsgemeinschaft (DFG) grant BR 6255/2-1. 
}

\maketitle

\begin{abstract}
We introduce the concept of ODD (`\textbf{O}rthogonally \textbf{D}egenerating on a \textbf{D}ivisor') Riemannian metrics on real analytic manifolds $M$. 
 These semipositive symmetric $2$-tensors may \emph{degenerate} on a finite collection of submanifolds, while their \emph{restrictions} to these submanifolds satisfy the inductive compatibility criterion to be an ODD metric again.

In this first in a series of articles on these metrics, we show that they satisfy basic properties that hold for Riemannian metrics. For example, we introduce orthonormal frames, the lowering and raising of indices, ODD volume forms and the Levi-Civita connection.
We finally show that an ODD metric induces a metric space structure on $M$ and that at least at general points of the \emph{degeneracy locus} $\mathcal{D}$, ODD vector fields are integrable and ODD geodesics exist and are unique.
\end{abstract}

\setcounter{tocdepth}{1} 
\tableofcontents

\section{Introduction}

The concept of Riemannian metrics on manifolds $M$ has seen several generalizations. Apart from the so-called \emph{Finsler metrics}, which are indeed norms on the tangent space $T_pM$, these are given by $2$-tensors with `relaxed properties'. The most important ones among them are probably the \emph{pseudo- (or semi-)Riemannian} and in particular \emph{Lorentz metrics}, which may be negative-definite in some direction but are everywhere nondegenerate (see e.g.~\cite[Section~2]{Lee18} for a discussion).
Another generalization, the \emph{sub-Riemannian} or \emph{Carnot–Carath\'eodory metrics} are defined only on a subbundle $S \subset TM$, see~\cite{Str86}, usually with some compatibility condition ensuring that points can be connected by curves tangent to $S$. All these generalizations have motivations from certain fields, such as physics, control theory, etc..

\subsection*{A motivation from birational geometry: singular K\"ahler metrics} In the present work, we introduce a new (at least to our knowledge) generalization of Riemannian metrics, which we call \emph{ODD metrics}.  Here, ODD stands for `\textbf{O}rthogonally \textbf{D}egenerating on a \textbf{D}ivisor', which stems from the initial motivation for defining these metrics: K\"ahler cone metrics with \emph{cone angle} $\beta > 2\pi$, called \emph{ramifold metrics} in~\cite[Section~1.1]{RT11}.

K\"ahler metrics with cone angle \emph{$\beta < 2\pi$} along a \emph{smooth divisor} $D$ are far better understood and in fact important in the proof of the Yau-Tian-Donaldson conjecture as is evident from the titles of the works~\cite{KE1,KE2,KE3}.

A natural direction in complex birational geometry is to generalize statements about complex manifolds to well-behaved singular varieties, so called \emph{klt pairs} $(X,\Delta)$. K\"ahler-Einstein metrics for example generalize to \emph{singular K\"ahler-Einstein metrics} introduced in~\cite{EGZ09}, where also the existence for $c_1(X,\Delta)\leq 0$ was proven. The equivalence of existence to \emph{log K-polystability} in the case $c_1(X,\Delta)>0$ and thus the singular version of the Yau-Tian-Donaldson conjecture was recently proven in~\cite{LXZ22}. 

While the existence is thus now settled, the \emph{singularities} of these metrics are still not fully understood. The approach goes via solving a complex Monge-Amp\`ere equation on a log-resolution $Y \to X$ with a possibly \emph{degenerate and singular} right hand side \emph{along the exceptional divisor}. This exceptional divisor is not smooth, but at least has simple normal crossings and the zeros and poles in the MA-equation correspond to cone angles that may be greater than $2\pi$. Thus in contrary to the smooth case, in the singular case one has to deal with intersections of the prime divisors \emph{and} with cone angles $>2\pi$.
In addition, further singularity may be introduced where the MA-equation degenerates. 

Such equations and K\"ahler metrics arise very naturally on log-resolutions of klt spaces in more general or entirely different settings (cscK and extremal metrics, curvature bounds and fundamental groups, etc.) and thus a better understanding of their behaviour would be very important.

\subsection*{The differential-geometric approach: ODD metrics}
In the present work, we initiate the study of metrics with \emph{cone angle} $\beta > 2\pi$ from the purely differential-geometric side. In order to do so, it makes sense to generalize the notion a little bit to the following (a slightly simplified version of Definition~\ref{def:ODD-metric}):

\begin{introdef}
Let $M$ be a real analytic manifold of dimension $n$. An \emph{ODD metric} $\godd$ on $M$ is an analytic section of $\mathrm{Sym}^2(T^*M)$, such that:
\begin{enumerate}
    \item $\godd_p$ is positive semidefinite for all $p \in M$.
    \item There is a finite collection $N_1,\ldots,N_m$, $m \in \mathbb{N}$ of closed analytic submanifolds of  $M$ of \emph{strictly smaller dimension}, such that $\godd_p$ is nondegenerate for $p \in M \, \setminus \bigcup_j N_j$.
    \item For each $N_j$, $1\leq j \leq m$, the symmetric bilinear form $\godd_p^{N_j}$ induced by $\godd_p$ on $TN_j$ is again an ODD Riemannian metric. 
\end{enumerate}
\end{introdef}

A few remarks are in order. First, note that we cannot demand the restrictions $\godd_p^{N_j}$ to be classical Riemannian since  in general, they will at least degenerate at intersections with other $N_i$. Second, Definition~\ref{def:ODD-metric} differs from the above in requiring that the $N_j$ have simple normal crossings. While for example this condition is apparently not needed for an induced metric space structure on $M$, we expect it to become important for statements from advanced Riemannian Geometry.
Third, we require analyticity because we will naturally encounter real-meromorphic and algebroid functions. Note that in particular complex manifolds are real-analytic.

\subsection*{ODD vector fields and forms}
We proceed in Section~\ref{sec:oddmetrics} with the definitions of ODD orthonormal frames, vector fields, forms, raising and lowering of indices, etc.. Here, we will give an intuitive account by considering the maybe simplest possible example.

The ODD metric $\godd= x^2 \, \mathrm{dx}^2$ on $M=\RR$ degenerates at the origin. For this metric, starting with the frame $\del_x=\frac{\del}{\del x}$, we obtain an \emph{ODD orthonormal frame} $E$ just by \emph{normalizing $\del_x$ with respect to $\godd$}, which yields
$$
E:=\frac{\del_x}{|\del_x|_\godd}=\frac{1}{\sqrt{x^2}}\del_x=\frac{1}{|x|}\del_x.
$$
This is not even a real-meromorphic, but an \emph{algebroid} vector field. Thus, we define \emph{ODD vector fields} $X$ with respect to $\godd$ to be those algebroid vector fields on $M$ that are of the form $X=f \, E$ with analytic $f$. Consequently, \emph{ODD one-forms} are those that can be written as $\omega= g \varepsilon$ with analytic $g$ in the \emph{ODD conormal frame} $\varepsilon=|x| \, \mathrm{dx}$ defined by $\varepsilon(E)=1$.
In the standard frame, the raising and lowering of indices are thus given by
$$
f \, \del_x \mapsto x^2 f \, \mathrm{dx},
\quad
g \, \mathrm{dx} \mapsto \frac{1}{x^2} g \, \del_x,
$$
for analytic $f$ and $g$, and map ODD vector fields to ODD one-forms and vice versa. We can further define higher differential forms, in particular an \emph{ODD volume form}
$$
dV_\godd=\varepsilon_1 \wedge \ldots \wedge \varepsilon_n =\sqrt{\det(\godd_{ij}} \, \mathrm{dx}_1 \wedge \ldots \wedge \mathrm{dx}_n
$$
for an orthonormal coframe $(\varepsilon_i)_i$, which in our example amounts to $dV_\godd = \varepsilon = |x| \mathrm{dx}$. In particular, we see that compact regular domains will have nonzero volume with respect to this measure. Moreover, we are able to define ODD versions of the gradient, divergence, and Laplacian.

\subsection*{ODD curves and the ODD distance}
In Section~\ref{sec:metricspace}, we first define what \emph{ODD regular curves} should be, namely continuous curves $\gamma \colon [a,b] \to M$, such that for some partition $a<a_1<\ldots<b$, the restrictions $\left.\gamma\right|_{(a_j,a_{j+1})}$ are analytic, and 
$$
        \lim_{t \nearrow a_j} \frac{\dot{\gamma}(t)}{|\dot{\gamma}(t)|_\godd} = \lim_{t \searrow a_j} \frac{\dot{\gamma}(t)}{|\dot{\gamma}(t)|_\godd}.
$$
As we show in Example~\ref{ex:RegularCurveODDvsClass}, an ODD regular curve through the origin in our example case is given by
$$
\gamma \colon [-1,1] \to M=\RR; \quad t \mapsto \mathrm{sgn}(t)\sqrt{|t|}.
$$
It is easy to see in this one-dimensional example, but indeed true in full generality - see Proposition~\ref{prop:ODDcurvesareregular} - that there is always a reparametrization of an ODD regular curve which is a classical regular curve. Here

It is now straightforward to use ODD regular curves to define an \emph{ODD distance function} $d_\godd$ on  $M$ and we can finally prove: 

\begin{introthm}[Theorem~\ref{thm:metricspace}]
Let $(M,\godd)$ be a connected ODD Riemannian manifold and $d_\godd$ its distance function. Then $(M,d_\godd)$ is a metric space and the metric topology associated to $d_\godd$ is the same as the topology of $M$ as a manifold.
\end{introthm}

\subsection*{Integrability of ODD vector fields and existence of ODD geodesics}

In the last section of the present work, we first define the \emph{ODD Levi-Civita connection}, Christoffel symbols etc. in analogy to the classical case. In our running example we compute $\Gamma_{11}^1=\frac{1}{x}$. In particular, $\nabla_E E = 0$ but $\nabla_{\del_x} \del_x = \frac{1}{x} \del_x$ is not even an ODD analytic vector field.
We define ODD vector fields along curves, where we have to be careful in the case of curves inside the degeneracy locus. We see that at least a covariant derivative $D_t$ along an ODD curve always gives a well-defined vector field along the curve.

Then we turn to investigate integrability of \emph{ODD vector fields}. We first show in Example~\ref{ex:nonuniqueintegralcurves}, that integral curves through a point $p \in M$ of an ODD vector field $X$ may be \emph{nonunique}. However, we prove the following:

\begin{introthm}[Theorem~\ref{thm:integrability-odd-vf}]
Let $(M,\godd)$ be an ODD Riemannian manifold and $p \in M$. Let $X$ be an ODD vector field. If either
\begin{enumerate}
    \item $p$ is a general point of $M$, or
    \item $p$ is a general point of some $N_j \subseteq \mathcal{D}$, which is maximal in the sense that it is not contained in any other $N_k$,
\end{enumerate}
then, there exists $ \epsilon>0$ and a unique ODD integral curve $\gamma\colon (-\epsilon,\epsilon) \to M$ with $p(0)=p$.
\end{introthm}

We also conjecture (Conj.~\ref{conj:int-vf}) that the \emph{existence} of integral curves holds for every $p \in M$ and \emph{even uniqueness holds} in the sense that if $\gamma$ and $\mu$ are two integral curves through $p$, then 
$$
\lim_{t \to 0} \frac{\dot{\gamma}(t)}{|\dot{\gamma}(t)|_\godd} \neq \lim_{t \to 0} \frac{\dot{\mu}(t)}{|\dot{\mu}(t)|_\godd}.
$$
Moreover, we sketch a possible proof of this conjecture in the follow-up of Section~\ref{sec:int-vf}. Along very similar lines we can prove the existence of \emph{ODD geodesics}, i.e. ODD curves with $D_t \dot{\gamma} =0$, at general points of the degeneracy locus:

\begin{introthm}[Theorem~\ref{thm:existence-geodesics}]
Let $(M,\godd)$ be an ODD Riemannian manifold and $p \in M$. Let $v \in T_pM$ be a tangent vector at $p$. If either
\begin{enumerate}
    \item $p$ is a general point of $M$, or
    \item $p$ is a general point of some $N_j \subseteq \mathcal{D}$, which is maximal in the sense that it is not contained in any other $N_k$,
\end{enumerate}
then, there exists a unique ODD geodesic $\gamma$ through $p$ in direction $v$.
\end{introthm}

We finally conjecture (Conj.~\ref{conj:existence-geodesics}) that unique geodesics exist for all points $p \in M$ and tangent directions $v \in T_pM$, where we expect a possible proof to look similar to the sketch of the proof of Conjecture~\ref{conj:int-vf}.

\subsection*{Further directions}

Since the ultimate application of ODD metrics we have in mind is in K\"ahler geometry of klt spaces, the following directions will be adressed in further works on ODD metrics:

\begin{enumerate}
    \item Give rigorous proofs of Conjectures~\ref{conj:int-vf} and~\ref{conj:existence-geodesics} and develop large parts of classical Riemannian geometry for ODD metrics, as far as e.g. the \emph{Bishop-Gromov volume comparison} or the \emph{Margulis Lemma}.
    \item Carry over the concept of ODD metrics from manifolds to \emph{orbifolds}. An envisioned ODD orbifold metric should be given in orbifold charts by ODD Riemannian metrics compatible with each other \emph{and} the orbifold structure. In particular, the degeneracy locus of the metric and the codimension-one ramification locus of the orbifold can agree, which yields \emph{arbitrary cone angles} $\beta \in \QQ_{>0}$.
    \item Finally, show that solutions to complex Monge-Amp\`ere equations with degenerate and singular right hand side are - maybe at least under additional assumptions -  of ODD type. 
\end{enumerate}

\section{ODD Riemannian Metrics}
\label{sec:oddmetrics}

Let $M$ be a  $C^\omega$ (i.e. real-analytic) manifold. The necessity of analyticity stems from the simple fact that we can define real-meromorphic (and more generally \emph{algebroid}) functions on $M$. We note here that if a Riemannian metric $g$ on a \emph{smooth} manifold is $C^\omega$ in every chart, then the manifold is already real-analytic due to the work~\cite{DK81}. Moreover, in the same article it is shown that there exists an atlas for which $g$ is real-analytic if and only if $g$ is real analytic in \emph{harmonic coordinates}.

For an open subset $U \subseteq M$, we denote by $C^\omega(U)$ the real-analytic functions $f \colon U \to \RR$. By $\ReMer(U)$, we denote the `real-meromorphic functions' - the field of fractions of the domain $C^\omega(U)$. Finally, by $\Alg(U)=\ReMer(U)^{\mathrm{alg}}$, we denote the \emph{algebroid functions} on $U$, the algebraic closure of the field of meromorphic functions.

\begin{remark}
\label{rem:algebroid}
{\em
An algebroid function may of course be multivalued and some values may be complex. In the following, we will be particularly interested in real branches of algebroid functions. For example, if $U=M=\RR$ and $f$ is a strictly positive analytic function, then as usual, by $\sqrt{f}$, we mean the positive branch of the algebroid function solving $T^2-f$, which is in fact again analytic. 
Another example is the real absolute value function $x \mapsto |x|$, which consists of the positive branches of the solution to $T^2-x^2$. Here, since $x^2$ is only semipositive, we have a branch point and thus four different continuous solutions $x,-x,|x|,$ and $-|x|$, of which only two are analytic.

One can construct algebroid functions in such a way that e.g. $\sqrt{f}$ and $-\sqrt{f}$ will be distinct functions, see~\cite[Sec.~1.1]{Er82}.
}
\end{remark}

We first give the definition of an ODD Riemannian metric, which is a special form of a \emph{semipositive} tensor.

\begin{definition}
\label{def:ODD-metric}
Let $M$ be a real analytic manifold of dimension $n$. An \emph{ODD Riemannian metric} $\godd$ on $M$ is an analytic section of $\mathrm{Sym}^2(T^*M)$, such that for the induced symmetric bilinear form $\godd_p \colon T_pM \times T_p M \to \mathbb{R}$ defined on each tangent space at $p \in M$, the following hold:
\begin{enumerate}
    \item $\godd_p$ is positive semidefinite for all $p \in M$.
    \item There is a finite collection $N_1,\ldots,N_m$, $m \in \mathbb{N}$ of closed analytic submanifolds of  $M$ of \emph{strictly smaller dimension}, such that $\godd_p$ is nondegenerate for $p \in M \setminus \bigcup N_j$.
    \item For each $N_j$, $1\leq j \leq m$, the symmetric bilinear form $\godd_p^{N_j}$ induced by $\godd_p$ on $TN_j$ is again an ODD Riemannian metric. 
    \item The collection $N_1,\ldots,N_m$ is \emph{simple normal crossing}, i.e., around every point $m \in M$, there exists a chart $(x^i)_{i=1,\ldots,n}$, such that every $N_j$ is given by the vanishing of some of the coordinates:
    $$
    N_j=\{x^{i_{j1}}=0,\ldots,x^{i_{jn_j}}=0  \}.
    $$
\end{enumerate}
We call the pair $(M,\godd)$ an ODD Riemannian manifold. We denote by $\mathcal{D}:=\bigcup_{j=1}^m N_j$ the \emph{degeneracy locus}.
\end{definition}

\begin{remark}
{\em
We note that the above definition has an inductive nature.

If $M$ is a curve, then the definition boils down to say that $\godd$ may degenerate on a finite set of points. In general, we may call a point \emph{general} (with respect to the ODD metric), if $\godd$ is nondegenerate at $p$, and \emph{special}, if $\godd$ is degenerate at $p$. Then the set of general points is a dense open submanifold of $M$, and the set of special points $\mathcal{D}$ is a finite union of ODD Riemannian manifolds of strictly smaller dimension. We note that, as the following example shows, for $N_j \neq N_i$ such that the nontrivial intersection $O_{ij}=N_j\cap N_i$ is a submanifold of $N_j$ of strictly smaller dimension, it may or may not be contained in the set of special points of $N_j$.
}
\end{remark}

\begin{example}
Let $M:=\mathbb{R}^3$ and $\godd$ be defined by
$$
\godd(x,y,z):= \mathrm{dx}\odot\mathrm{dx} + \mathrm{dy}\odot\mathrm{dy} + (x^2+z^2)(y^2+z^2)\mathrm{dz}\odot\mathrm{dz}.
$$
Then $\godd$ is an ODD Riemannian metric. The set of special points is the union of the $y$- and the $x$-axis. It is clear that on both axes, $\godd$ induces a veritable Riemannian metric.

If instead we consider $\mathfrak{h}$, defined by
$$
\mathfrak{h}(x,y,z):= (x^2+z^2)\mathrm{dx}\odot\mathrm{dx} + \mathrm{dy}\odot\mathrm{dy} + (x^2+z^2)(y^2+z^2)\mathrm{dz}\odot\mathrm{dz},
$$
then $\mathfrak{h}$ again is an ODD metric with the same set of special points, but now the origin is a special point for the induced ODD metric on the $x$-axis, while on the $y$-axis, the induced metric is again veritable Riemannian. Note that the dimension of the degenerate subspace of $TM$ is two along the $y$-axis and one along the $x$-axis.
\end{example}

\subsection{ODD Orthonormal Frames}

It is clear that a priori, if for an ODD metric we try to produce an orthonormal frame from an arbitrary one by the Gram-Schmidt orthonormalization process, we will run into trouble as soon as some vector field lies in the degeneracy locus at some $N_j$. At a general point of $N_j^m$ not lying in any other $N_k$, starting with a coordinate frame $(\del_i=\frac{\del}{\del x^i})_{i=1,\ldots,n}$ for slice coordinates $(x^i)_{i=1\ldots,n}$, we could at least find a frame  $(X_i)_{i=1,\ldots,n}$, such that $(\left.X_i\right|_{N_j})_{i=1,\ldots,m}$ is an orthonormal frame on $N_j$, and $(\left.X_i\right|_{N_j})_{i=m+1,\ldots,n}$ span the degenerate subspace of $T_pM$ at every point $p \in N_j$.

But as soon as we encounter more special points, the frames produced in this way get less and less meaningful. Indeed, if $\godd$ is totally degenerate at a point $p$, this procedure will leave any input frame around $p$ as it is.  

Instead, in order to define orthonormal frames appropriately, we have to consider \emph{algebroid vector fields}. That is, sections of the sheaf $\Alg(TM) \supseteq C^\omega(TM)$, which is a sheaf of $\Alg(M)$-vector spaces. 

Starting with an arbitrary $C^\omega$ local coordinate frame $(\del_i=\frac{\del}{\del x^i})_{i=1,\ldots,n}$ defined on $U \subseteq M$, we can perform the steps in the Gram-Schmidt orthonormalization process as usual, where the $i$-th vector field 
$$
E_i:=\frac{\del_i-\sum_{j=1}^{i-1} \frac{\godd_{ij}}{\godd_{jj}}\del_j}{\left| \del_i-\sum_{j=1}^{i-1} \frac{\godd_{ij}}{\godd_{jj}}\del_j\right|_\godd}
$$
of the resulting orthonormal frame $(E_i)_{i=1,\ldots,n}$ is a well-defined algebroid vector field (compare Remark~\ref{rem:algebroid}), where $\godd_{ij}$ are the matrix coefficients of $\godd$ in the frame $(\del_i)$. In particular, such a vector field may have poles at $p \in U$, so there might not be an orthonormal basis $(\left.E_1\right|_p,\ldots, \left.E_n\right|_p)$ of $T_pM$ in general, but the identities 
$$
\godd(E_i,E_j) = \delta_{ij}
$$
hold in $\Gamma(U,\Alg(TM))$. This brings us straight to the definition.

\begin{definition}
Let $(M,\godd)$ be an ODD Riemannian manifold. Then an ODD orthonormal frame w.r.t. $\godd$ on $U \subseteq M$ is a collection $(E_i)_{i=1,\ldots,n}$ of sections of $\Alg(TU)$, such that
$$
\godd(E_i,E_j) = \delta_{ij}
$$
holds for all $1\leq i,j\leq n$. 
\end{definition}

The above considerations show that an ODD orthonormal frame always exists on any chart $U \subseteq M$. We also note that the condition $\godd(E_i,E_j) = \delta_{ij}$ automatically implies that the only poles of the $E_i$ may lie in the degenerate subspaces of $\godd$ and are of corresponding order.

We may similarly define ODD conormal frames. Expressed in a real-analytic coframe, these of course possess zeros instead of poles. Expressed in an ODD conormal frame, an ODD metric is (as usual) given by the identity matrix.
One can check that the following definition does not depend on the chosen ODD orthonormal frames, so it is well-defined:

\begin{definition}
\label{def:ODD-vectorfield}
Let $(M,\godd)$ be an ODD Riemannian manifold and $X$ ($\omega$) be an algebroid vector field (one-form). We say that $X$ ($\omega$) is an \emph{ODD vector field (one-form)}, if in any orthonormal frame $(E_i)_i$ (conormal frame $(\varepsilon^i)_i$), the coefficients $X^i$ ($\omega_i$) are real-analytic functions.
\end{definition}

We define analogously arbitrary ODD tensors.

\subsection{ODD Musical Isomorphisms and the Gradient}

As in the case of ordinary Riemannian metrics, we wish to define the lowering and raising of indices. First we note that we get a well-defined though in general not injective bundle homomorphism $\hat{\godd}\colon TM \to T^*M$ given by
$$
\hat{\godd}(v)(w):=\godd_p(v,w)
$$
for all $p \in M$ and $v,w \in T_pM$. Thus we may safely define the \emph{lowering of indices} $X  \mapsto X^\flat$ at least for real-analytic vector fields $X \in C^\omega(TM)$. However, if we extend the domain to algebroid sections of $TM$, in the image we encounter algebroid one-forms, elements of $\Alg(T^*M)$. 
In a local real-analytic frame over $U \subseteq M$, the coefficient matrix $\godd_{ij}$ is invertible over the field $\ReMer(U) \subseteq \Alg(U)$, as follows from the Gram-Schmidt-orthogonalization process. We denote the inverse matrix by $\godd^{ij}$ as usual, it has meromorphic entries. Of course, in ODD orthonormal (co-)frames, $\godd^{ij}$ will be the identity matrix as well. 

Now we can define the \emph{raising of indices} of a one-form $\omega$ given in a local coframe $(\epsilon^i)$ dual to the frame $(E_i)$ by $\omega=\omega_j \epsilon^j$ setting
$$
\omega^\sharp := \godd^{ij} \omega_j E_i.
$$
The raising and lowering of indices define inverse isomorphisms between the $\Alg(M)$-vector spaces of algebroid vector fields and one-forms. 

As in the classical theory, we use the raising of indices to define the \emph{gradient} of an analytic (or algebroid) function $f \colon M \to \RR $ by setting 
$$
\grad f := (df)^\sharp,
$$
which is in a local frame $(E_i)$ expressed as
$$
\grad f = (\godd^{ij} E_i f) E_j.
$$
We note that while - again as in the classical theory - in a local ODD orthonormal frame, the components of $\grad f$ are the same as those of $df$, when $df$ \emph{vanishes at some point in some real-analytic frame}, it may happen that it does not so in an ODD frame, and in particular, it may have a well defined gradient vector field. 

We extend the flat and sharp operators to tensors of arbitrary rank as usual.
Moreover, we can define an `ODD inner product' for covectors, which as one might guess, instead of being degenerate, will have poles. In terms of one-forms and in a usual local frame, it is given as
$$
\langle \omega, \eta \rangle := \godd(\omega^\sharp, \eta^\sharp ) = \godd^{ij} \omega_i \eta_j.
$$
We extend this to an `ODD inner product' for tensors of arbitrary rank.

\subsection{ODD Volume Forms}

We always assume $M$ to be oriented. We can directly adopt the definition of the ODD volume form associated to an ODD metric.

\begin{definition}
Let $(M,\godd)$ be an ODD Riemannian manifold. There is a unique algebroid $n$-form $dV_\godd$, the \emph{ODD volume form}, given by one of the following equivalent properties:
\begin{enumerate}
    \item In any local ODD orthonormal oriented coframe $(\epsilon_i)$, the ODD volume form is given by $$dV_\godd=\epsilon_1 \wedge \ldots \wedge \epsilon_n.$$
    \item For a local oriented chart $(x^i)$, the ODD volume form is given by
    $$
    dV_\godd= \sqrt{\det(\godd_{ij})} dx_1 \wedge \ldots \wedge dx_n.
    $$
\end{enumerate}
\end{definition}

Using the ODD volume form, we may define integrals of functions and in particular define the volume of compact regular domains $D \subseteq M$ by
$$
\mathrm{Vol}(M):= \int_D 1 \, dV_\godd.
$$
In particular, we see that since $dV_\godd$ vanishes only on a subset of codimension greater or equal to one, no compact regular domain will have volume zero, i.e. computing volumes with $dV_\godd$ \emph{makes sense}.  

\subsection{Divergence and Laplacian}

Similarly to the gradient, we can define the important differential operators \emph{divergence} and \emph{Laplacian}. For a vector field $X$ and a $p$-form $\omega$, we denote by $X \iprod \omega$ the \emph{interior product} (or contraction), a $(p-1)$-form given by
$$
X \iprod \omega(Y_1,\ldots,Y_{p-1}):= \omega(X,Y_1,\ldots,Y_{p-1})
$$
for vector fields $Y_j$. So the divergence $\divv X$ of a vector field is characterized by the formula
$$
d( X \iprod dV_\godd) = (\divv X) dV_\godd,
$$
while the Laplacian of a function $f$ is given by
$$
\Delta f = \divv (\grad f).
$$
Local coordinate representations are of the `usual form':
$$
\divv \left( X^i \del_i\right) = \frac{1}{\sqrt{\det(\godd_{kl})}} \del_i (X^i \sqrt{\det(\godd_{kl})}),
\qquad
\Delta f = \frac{1}{\sqrt{\det(\godd_{kl})}} \del_i \left(\godd^{ij} \sqrt{\det(\godd_{kl})} \frac{\del f}{\del x^j} \right).
$$

\section{ODD Riemannian Manifolds as Metric Spaces}
\label{sec:metricspace}

An important property of ODD Riemannian metrics is that despite their degeneracy, they still induce a metric on the underlying manifold. We prove this fact in the following.

\subsection{Curves and Lengths}

In this subsection, we introduce \emph{ODD piecewise regular curves} and their lengths.

\begin{definition}
\label{def:ODD-regular-curve}
Let $(M,\godd)$ be an ODD Riemannian manifold. Let $a,b \in \RR$ and $\gamma \colon [a,b] \to M$ be a continuous curve segment. 
\begin{enumerate}
    \item When $\gamma\colon (a,b) \to M$ is a real-analytic map, and $|\dot{\gamma}(t)|_{\godd}$ is bounded on $(a,b)$, we say that $\gamma$ is a \emph{regular curve} w.r.t. $\godd$.
    \item When $\gamma\colon [a,b] \to M$ is an algebroid function and there is a finite partition $a=a_0<a_1<\ldots<a_k=b$, such that:
    \begin{itemize}
        \item $\left. \gamma \right|_{[a_{j-1},a_{j}]}$ is a regular curve w.r.t. $\godd$,
        \item for every $0<j<k$, we have
        $$
        \lim_{t \nearrow a_j} \frac{\dot{\gamma}(t)}{|\dot{\gamma}(t)|_\godd} = \lim_{t \searrow a_j} \frac{\dot{\gamma}(t)}{|\dot{\gamma}(t)|_\godd},
        $$
        or equivalently 
         $$
        \lim_{t \nearrow a_j} \frac{\dot{\gamma}(t)}{|\dot{\gamma}(t)|_g} = \lim_{t \searrow a_j} \frac{\dot{\gamma}(t)}{|\dot{\gamma}(t)|_g},
        $$
        with respect to some local Riemannian metric $g$, then we call $\gamma$  an \emph{ODD regular curve} w.r.t. $\godd$.
    \end{itemize}
    \item If there is a finite partition $a=a_0<a_1<\ldots<a_k=b$, such that $\left. \gamma \right|_{[a_{j-1},a_{j}]}$ is an ODD regular curve for any $1\leq j \leq k$, then we say that $\gamma$ is an \emph{ODD piecewise regular curve}.
\end{enumerate}
\end{definition}

\begin{remark}
{\em
One can check that Item (2) is equivalent to saying that for every local ODD orthonormal frame $(E_i)$ on $M$, $\dot{\gamma}(t)$ can be expressed as $\dot{\gamma}(t)= \dot{\gamma}^i(t) \left.E_i\right|_{\gamma(t)}$ with analytic local functions on $[a,b]$. 
This could be interpreted as saying that $\dot{\gamma}$ is a well-defined \emph{ODD vector field along $\gamma$} (compare Def.~\ref{def:ODD-vf-along-curve}). The partition in Item (2) obviously comes from the intersections of $\gamma$ with the degeneracy locus $\mathcal{D}$.

Here we maybe see for the first time why in Definition~\ref{def:ODD-metric}, it is important to impose Condition (3), namely that $\godd$ does not degenerate tangentially to the degeneration locus. Otherwise, for any $\gamma$ lying \emph{in} some $N_i$, where $\godd$ degenerates tangentially, it would not be possible to express $\dot{\gamma}$ in this way.
}
\end{remark}

\begin{example}
\label{ex:RegularCurveODDvsClass}
 We consider $M:=\RR$ together with the ODD metric $\godd:= x^2 \mathrm{dx}^{\odot 2}$. As an example of an ODD regular curve that is not regular in the classical sense, the curve 
$
\gamma \colon [-1,1] \to M
$ with the two segments
$$
\left.\gamma\right|_{[-1;0]}\colon  t \mapsto -(-t)^{1/2},
\qquad
\left.\gamma\right|_{[0;1]}\colon  t \mapsto t^{1/2},
$$
is continuous, but not piecewise regular in the classical sense of~\cite{Lee18}, since $\lim_{t \to 0} \dot{\gamma}(t) = \infty$. In fact, we can see that $\gamma$ is an algebroid function solving $\gamma^4=t^2$. Moreover, computing the norm of the derivative for $t \in [0,1]$ yields
$$
|\dot{\gamma}(t)|_\godd = \sqrt{ (t^{\frac{1}{2}})^2 \left(\frac{1}{2}  t^{-\frac{1}{2}}\right)^2}=\frac{1}{2},
$$
while for $t \in [-1,0]$ we get
$$
|\dot{\gamma}(t)|_\godd = \sqrt{ (- (-t)^{\frac{1}{2}})^2 \left(\frac{1}{2}  (-t)^{-\frac{1}{2}}\right)^2}=\frac{1}{2}
$$
as well. This yields that $\gamma$ has \emph{constant speed} with respect to $\godd$. Now for the (global) orthonormal frame $(E:=\frac{\del_x}{|x|})$, we have
$$
\dot{\gamma}(t)=\frac{E(\gamma(t))}{2},
$$
so $\dot{\gamma}$ is an ODD vector field by Definition~\ref{def:ODD-vectorfield} and thus $\gamma$ is an ODD regular curve.

\end{example} 

\begin{definition}
Let $(M,\godd)$ be an ODD Riemannian manifold and $\gamma \colon [a,b] \to M$ be an ODD piecewise regular curve. Then an \emph{ODD reparametrization} of $\gamma$ is an ODD piecewise regular curve $\tilde{\gamma}=\gamma \circ \phi \colon [c,d] \to M$, such that
\begin{enumerate}
    \item $\phi\colon [c,d] \to [a,b]$ is a homeomorphism,
    \item there is a finite partition $c=c_0<c_1<\ldots<c_l=d$, such that $\left. \phi \right|_{(c_{j-1},c_{j})}$ is an analytic diffeomorphism on its image. 
\end{enumerate}
\end{definition}

\begin{example}
For the curve $\gamma$ from Example~\ref{ex:RegularCurveODDvsClass}, the obvious classically regular parametrization $\tilde{\gamma}:=id \colon [-1,1] \to M$ is an ODD reparametrization, since it is of the form  $\tilde{\gamma}= \gamma \circ \phi$, with 
$$
\left.\phi\right|_{[-1;0]}\colon  t \mapsto -t^2,
\qquad
\left.\gamma\right|_{[0;1]}\colon  t \mapsto t^{2},
$$
and the restrictions of $\phi$ to the \emph{open intervals} $(-1,0)$ and $(0,-1)$ are analytic diffeomorphisms. However, $\gamma$ can not be a classical reparametrization of $\tilde{\gamma}$.
\end{example}

By Definition~\ref{def:ODD-regular-curve}, the (piecewise) regular curves from classical Riemannian geometry (see e.g.~\cite{Lee18}) are in particular ODD (piecewise) regular curves. But ODD curves are more general, since we allow the derivative $\dot{\gamma}$ to diverge at interval boundaries, as long as the $\godd$-norm $|\dot{\gamma}(t)|_{\godd}$ stays bounded. However, as the following statement shows, there is always a reparametrization (and repartition) of an ODD piecewise regular curve, that is a classical piecewise regular curve.

\begin{proposition}
\label{prop:ODDcurvesareregular}
Let $(M,\godd)$ be an ODD Riemannian manifold. Let $a,b \in \RR$ and $\gamma \colon [a,b] \to M$ be a continuous curve segment. Then the following hold.
\begin{enumerate}
    \item If $\gamma$ is an ODD regular curve w.r.t. $\godd$, there is a reparametrization which is a classical regular curve.
    \item If $\gamma$ is an ODD piecewise regular curve w.r.t. $\godd$,  there is a reparametrization which is a classical piecewise regular curve.
\end{enumerate}
\end{proposition}

\begin{proof}
We begin with Item (1). We can assume that $\gamma$ lies in some coordinate chart and that $\dot{\gamma} \neq 0$ on $[a,b]$. For some Riemannian metric $g$ on the coordinate chart, we can solve the following differential equation on every open interval $(a_{j-1},a_j)$, where $\gamma$ is real analytic. 
$$
\dot{\phi_j}(t)=\frac{1}{|\dot{\gamma}(t)|_g}
$$
 The solution $\phi_j$ is analytic, since the right hand side is analytic on $(a,b)$. So $\gamma \circ \phi$ is a reparametrization of $\left.\gamma\right|_{[a_j,a_{j+1}]}$ with constant speed. Due to Def.~\ref{def:ODD-regular-curve} (2), we have
  $$
        \lim_{t \nearrow a_j} \frac{\dot{\gamma}(t)}{|\dot{\gamma}(t)|_g} = \lim_{t \searrow a_j} \frac{\dot{\gamma}(t)}{|\dot{\gamma}(t)|_g},
        $$
        so $\gamma \circ \phi$ is regular on the whole interval $[a,b]$. 
 
 For Item (2), let  $\gamma$ be an ODD piecewise regular curve. Then there is a finite partition into ODD regular curve segments, so Item (2) follows from Item (1).
\end{proof}

The following is now well-defined as in the classical case.

\begin{definition}
Let $(M,\godd)$ be an ODD Riemannian manifold and $\gamma \colon [a,b] \to M$ be an ODD piecewise regular curve. We define the \emph{length of $\gamma$} to be
$$
L_\godd(\gamma) := \int_a^b |\dot{\gamma}(t)|_\godd dt.
$$
\end{definition}

It is clear that properties of the length function such as additivity and invariance under ODD reparametrization and ODD isometries follow in the same way as in the classical case from the linearity of the integral and the chain rule. Also, it is clear that any ODD piecewise regular curve has a unique forward parametrization by arc length (compare~\cite[Prop.~2.49]{Lee18}).

\subsection{The Distance Function}

On a connected ODD Riemannian manifold $(M,\godd)$, by the same arguments as in the classical case, we know that any pair of points $p,q \in M$ can be joined by an ODD piecewise regular curve. We define the ODD Riemannian distance $d_\godd(p,q)<\infty$ to be the infimum of all lengths of ODD piecewise regular curves joining $p$ and $q$. 
In the following, we prove that $(M,d_\godd)$ is a metric space, and the induced topology is the same as the manifold topology of $M$. 
Here we note that we cannot compare lengths of tangent vectors with respect to $\godd$ and a classical Riemannian metric $g$ in the sense that there exist positive constants $c$ and $C$ with
$$
c|v|_g \leq |v|_\godd \leq C|v|_g
$$
as usual. In particular, we do not have an analogue of~\cite[Lemma~2.53]{Lee18}. But we can prove an analogue of the crucial~\cite[Lemma~2.54]{Lee18} with other methods.

\begin{lemma}
\label{le:comparewithRiemMetric}
Let $(M,\godd)$ be an ODD Riemannian manifold and $d_\godd$ its distance function. Let  $p \in M$. Then there is an open coordinate neighbourhood $p \in U$ and an open $p \in V \subseteq U$   and positive constants $C$ and $D$, such that the following hold:

\begin{enumerate}
\item If $q \in V$, then $d_\godd(p,q)\leq C d_g(p,q)$, where $g$ is the Euclidean metric in $V$.
 \item If $q \notin V$, then $d_\godd(p,q) \geq D$.
\end{enumerate}

\end{lemma}

\begin{proof}
We choose a small Euclidean ball $V:=B_\epsilon(p)$ around $p$ in some coordinate neighbourhood $U$, such that we can assume that the $N_j$ are given by the vanishing of coordinates and $p$ is the origin. 

We prove Item (1). Let $q \in V$. Since $\overline{V}$ is compact and $\godd$ is semipositive, but may degenerate somewhere in $V$, we have at least the first inequality from~\cite[Lemma~2.53]{Lee18}:
$$
|v|_\godd \leq C |v|_g
$$
for some positive constant $C$ and all classical tangent vectors $t \in T_{p'}M$ for all $p' \in \overline{V}$. Now let $\gamma$ be a regular parametrization of the line segment from $p$ to $q$. In particular, $\gamma$ is an ODD regular curve. Thus due to the above inequality, we have
$$
d_\godd(p,q) \leq L_\godd(\gamma) \leq C L_g(\gamma) = C d_g(p,q)
$$
and Item (a) is proven.

Now in order to prove (b), let $q \notin V$ and $\gamma$ be an ODD piecewise regular curve from $p$ to $q$. We can safely assume that $q \in \del V$ by if necessary replacing it with $\gamma(\mathrm{inf}(t;\gamma(t) \notin V))$.

Now let $K':=\{x ~|~ |x_i| \leq \epsilon/2 \forall 1\leq i \leq n \} \subseteq V$ be a small box around $p=0$ and let $K_j':=K' \cap \{x ~|~ |x_j| \geq \epsilon/4\}$ for $j=1,\ldots,n$. We can assume that $K \cap \gamma$ is connected and thus $\gamma$ has a unique leaving point $\{s\}=K' \cap \gamma$. In particular, there is $j \in \{1,\ldots,n\}$, such that the $j$-th coordinate of $s$ is equal to $\epsilon/2$. The situation is illustrated in the picture below.

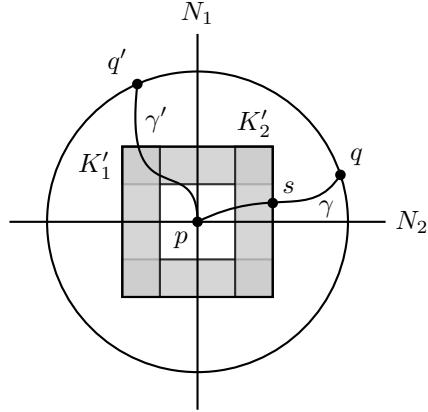
\begin{center}
\begin{tikzpicture}

\draw[draw=black,fill=black!20, opacity=0.8] (-1,0.5) rectangle (1,1);
\draw[draw=black,fill=black!20, opacity=0.8] (-1,-1) rectangle (1,-0.5);
\coordinate[label=180:$K_1'$] (K) at (-1,0.75);

\draw[draw=black,fill=black!20, opacity=0.8] (0.5,-1) rectangle (1,1);
\draw[draw=black,fill=black!20, opacity=0.8] (-1,-1) rectangle (-0.5,1);
\coordinate[label=90:$K_2'$] (K) at (0.75,1);

\coordinate[label=-135:$p$] (P) at (0,0);
\fill (P) circle (2pt);
\coordinate[label=45:$q$] (Q) at (1.9,0.6245);
\fill (Q) circle (2pt);

\coordinate[label=135:$q'$] (Q') at (-0.8,1.833);
\fill (Q') circle (2pt);

\coordinate[label=45:$s$] (S) at (1,0.25);
\fill (S) circle (2pt);

\draw (Q') .. controls (-1,0) and (0,1) .. node[right, pos=0.1] {$\gamma'$} (P);

  \draw (-2.5,0) -- (2.5,0);
  \draw (0,-2.5) -- (0,2.5);
  \draw (0,0) circle [radius=2cm];
  \draw (-1,-1) rectangle (1,1);
  
  \coordinate[label=90:$N_1$] (X) at (0,2.5);
  \coordinate[label=0:$N_2$] (Y) at (2.5,0);

\draw (P) .. controls (1,0.5) and (1.5,0) .. node[below, pos=0.85] {$\gamma$} (Q);  
\end{tikzpicture}
\end{center}

The crucial point is that $\godd$ does not degenerate on $K_j'$ in $x_j$-direction and since $K_j'$ is compact and the entry point of $\gamma$ in $K_j'$ has it's $j$-th  equal to $\epsilon/4$, there exist $D_j$, $j=1,\ldots,n$,  such that $L_\godd(\gamma)\geq D_j$, \emph{if} the point $s$ where $\gamma$ leaves $K'$ has the $j$-th coordinate equal to $\epsilon/2$. Taking $D:=\min_{j=1,\ldots,n} D_j$ yields the conclusion.

\end{proof}

\begin{theorem}
\label{thm:metricspace}
Let $(M,\godd)$ be a connected ODD Riemannian manifold and $d_\godd$ its distance function. Then $(M,d_\godd)$ is a metric space and the metric topology associated to $d_\godd$ is the same as the topology of $M$ as a manifold.
\end{theorem}

\begin{proof}
The proof is verbatim to the proof of~\cite[Theorem~2.55]{Lee18}, with ~\cite[Lemma~2.54]{Lee18} replaced by our Lemma~\ref{le:comparewithRiemMetric}.
 
\end{proof}

\section{The ODD Levi-Civita Connection, integrability of vector fields and geodesics}

\subsection{The ODD Levi-Civita Connection}

The ODD Levi-Civita connection will be a particular example of a \emph{real-meromorphic connection}, a notion that we will define first. We note that related notions of meromorphic connections can be found in the literature, cf.~\cite[Ch.~5]{HTT08}.

\begin{definition}
Let $M$ be a real-analytic manifold, $E \to M$ an analytic vector bundle, and $\Alg(TM)$ and $\Alg(E)$ be the $\Alg(M)$-vector spaces of algebroid vector fields and algebroid sections of $E$.
An \emph{algebroid connection} in $E$ is a map
$$
\nabla \colon  \Alg(TM) \times \Alg(E) \to \Alg(E); \quad (X,Y) \mapsto \nabla_X Y,
$$
with the following properties:
\begin{enumerate}
    \item $\nabla_X Y$ is $\Alg(M)$-linear in $X$ and $\RR$-linear in $Y$.
    \item For $f \in \Alg(M)$, we have
    $$
    \nabla_X (fY)= f \nabla_X Y + (Xf)Y
    $$
\end{enumerate}
\end{definition}

Similarly to the case of classical Koszul connections, compare~\cite[Le.~4.1]{Lee18}, it is clear that the 'value' of $\nabla_X Y$ at some $p \in M$ (either as a vector of $E_p$ or as a certain kind of 'pole') depends only on $X$ and $Y$ on small neighbourhoods of $p$. But in contrary to the classical case, if $X$ has a pole at $p$, we can not say that $\nabla_X Y$ depends only on $X$ at $p$. However, for the ODD Levi-Civita connection, this is true in a certain sense, cf. Remark~\ref{rem:ODDLV-conn-locality}

In the following, we will consider connections in the tangent bundle. As usual, in an analytic local frame $(E_i)$, we have the connection coefficients $\Gamma_{ij}^k$ satisfying
$$
\nabla_{E_i} E_j = \Gamma_{ij}^k E_k,
$$
but here the $\Gamma_{ij}^k$ are algebroid functions. As in the classical case, the connection is completely determined by the $\Gamma_{ij}^k$ locally, as for algebroid vector fields $X:=X^i E_i$ and $Y:=Y^i E_i$, we get
$$
\nabla_X Y = (X(Y^k) X^iY^j +  \Gamma_{ij}^k)E_k
$$
exactly in the same way as in the proof of~\cite[Prop.~4.6]{Lee18}. Also in the same way, we have a transformation law for the $\Gamma_{ij}^k$. When $(E_i)$ and $(\tilde{E}_j)$ are two analytic frames on some open $U \subseteq M$, related by $\tilde{E}_i=A_i^j E_j$ for some (analytic!) matrix of functions $A_i^j$, and letting $\Gamma_{ij}^k$ and $\tilde{\Gamma}_{ij}^k$ be the connection coefficients with respect to $(E_i)$ and $(\tilde{E}_j)$ respectively, then 
$$
\tilde{\Gamma}_{ij}^k = (A^{-1})_p^kA_i^qA_j^r \Gamma_{pq}^r + (A^{-1})_p^kA_i^q E_q(A_j^p).
$$
Moreover, as in the classical case, a connection in $TM$ extends to arbitrary tensor fields, compare~\cite[Prop.~4.15]{Lee18}. So we can define the \emph{total covariant derivative} of a $(k,l)$-tensor field $F$ to be the $(k,l+1)$-tensor field defined by
$$
(\nabla F)(\omega^1,\ldots,\omega^k,Y^1,\ldots,Y^k,X):=(\nabla_X F)(\omega^1,\ldots,\omega^k,Y^1,\ldots,Y^k)
$$
for arbitrary $1$-forms $\omega^j$ and vector fields $Y^i,X$. 
Finally, we can define the covariant derivative of vector fields along analytic curves as usual, but we can not guarantee the existence of meaningful geodesics for a general algebroid connection. We have to first define the Levi-Civita connection of an ODD metric, then we can prove the existence of ODD geodesics.

First, when $(M,\godd)$ is an ODD Riemannian manifold, we call an algebroid connection $\nabla$ an \emph{ODD metric connection}, if for all $X,Y,Z \in \ReMer(TM)$, it satisfies
$$
\nabla_X \langle Y, Z \rangle = \langle \nabla_X Y, Z \rangle + \langle Y, \nabla_X Z \rangle.
$$
As in the classical case, this is equivalent to the total covariant derivative of $\godd$ being trivial. Moreover, we say that an algebroid connection is \emph{symmetric}, if the torsion tensor
$$
\tau(X,Y):=\nabla_X Y - \nabla_Y X - [X,Y], 
$$
where $[X,Y]$ is the Lie bracket of vector fields, vanishes identically on $M$ for all $X,Y \in \Alg(TM)$. As in the classical case, it follows that $\nabla$ is symmetric, if and only if the connection coefficients in every \emph{coordinate} frame satisfy $\Gamma^k_{ij}=\Gamma^k_{ji}$.

\begin{theorem}[Fundamental Theorem of ODD Riemannian geometry]
Let $(M,\godd)$ be an odd Riemannian manifold. There exists a unique metric and symmetric connection $\nabla$ in $TM$. We call it the ODD Levi-Civita connection of $\godd$.
\end{theorem}

\begin{proof}
We follow the lines of the classical proof, first proving uniqueness.
We assume that $\nabla$ is such a metric and symmetric connection in $TM$ and let $X,Y,Z \in \Alg(TM)$. Writing the metric condition with $X,Y,Z$ in different orders, we get 
\begin{align*}
    X \langle Y, Z \rangle &= \langle \nabla_X Y, Z \rangle + \langle Y,  \nabla_X Z \rangle ,  \\
    Y \langle Z, X \rangle &= \langle \nabla_Y Z, X \rangle + \langle Z,  \nabla_Y X \rangle , \\
    Z \langle X, Y \rangle &= \langle \nabla_Z X, Y \rangle + \langle X,  \nabla_Z Y \rangle .
\end{align*}
Now we apply the symmetric condition to rewrite the last term on the right in every equation and get
\begin{align*}
    X \langle Y, Z \rangle &= \langle \nabla_X Y, Z \rangle + \langle Y,  \nabla_Z X \rangle + \langle Y, [X,Z]\rangle , \\
    Y \langle Z, X \rangle &= \langle \nabla_Y Z, X \rangle + \langle Z,  \nabla_X Y \rangle + \langle Z, [Y,X]\rangle , \\
      Z \langle X, Y \rangle &= \langle \nabla_Z X, Y \rangle + \langle X,  \nabla_Y Z \rangle + \langle X, [Z,Y]\rangle .
\end{align*}
Adding the first two equations, subtracting the third and solving for $\langle \nabla_X Y, Z \rangle$, we arrive at
$$
\langle \nabla_X Y, Z \rangle = \frac{1}{2}\left( X \langle Y, Z \rangle + Y \langle Z, X \rangle - Z \langle X, Y \rangle - \langle Y, [X,Z]\rangle - \langle Z, [Y,X]\rangle + \langle X, [Z,Y]\rangle \right).
$$
The point here is that the right hand side of this equation does not depend on $\nabla$ and the lowering of indices is an isomorphism (in particular, it is injective). Thus, for another metric and symmetric connection $\tilde{\nabla}$, the above equation yields
$$
(\tilde{\nabla}_X Y)^\flat = (\nabla_X Y)^\flat
$$
and ergo, both connections are identical. Now as in the classical case, in order to prove existence, we use our equation obtained above rewritten in an analytic coordinate chart $(x^i)$. Since the Lie brackets of the coordinate vector fields $\del_i$ are zero, we obtain 
$$
\langle \nabla_{\del_i} \del_j, \del_l \rangle = \frac{1}{2}\left( \del_i \langle \del_j, \del_l \rangle + \del_j \langle \del_l, \del_i\rangle - \del_l \langle \del_i, \del_j\rangle \right).
$$
By the definitions of the $\godd_{ij}$ and the $\Gamma_{ij}^k$, this yields
$$
\Gamma_{ij}^m \godd_{ml} = \frac{1}{2}(\del_i\godd_{jl}+ \del_j \godd_{li} + \del_l \godd_{ij}).
$$
Multiplying with the inverse $\godd^{kl}$ and noting that $\godd_{ml}\godd^{kl} = \delta_m^k$, we get the formulae
$$
\Gamma_{ij}^k = \frac{1}{2}\godd^{kl}(\del_i\godd_{jl}+ \del_j \godd_{li} + \del_l \godd_{ij}),
$$
which are meromorphic functions and define an obviously symmetric algebroid connection in the chart. Analogously to the classical case, it can be shown that the connection also is metric. Uniqueness means that all metrics defined locally in this way in the charts agree on the overlapping subsets.
\end{proof}

As usual, we cannot assume an ODD orthonormal frame $(E_i)_i$ to be holonomic. In particular, the Lie brackets $[E_i,E_j]$ will not vanish in general. Thus, in such an ODD orthonormal frame, the formulae for the $\Gamma_{ij}^k$ are given by
$$
\Gamma_{ij}^k = \frac{1}{2}(c_{ij}^k-c_{ik}^j-c_{jk}^i),
$$
where the $c_{ij}^k$ are the coefficients in the expansion
$$
[E_i,E_j]=c_{ij}^k E_k
$$
of the Lie bracket. 

\begin{remark}
\label{rem:ODDLV-conn-locality}
For ODD vector fields $X,Y$ on an ODD Riemannian manifold and a local orthonormal frame $(E_i)$, obviously the value of $\nabla_X Y$ at $p \in M$ only depends on the values $X^i(p)$ of the coefficient functions, where $X=X^i E_i$. So in a sense, this locality is the same as in the classical case. However, if $E_i$ has a pole at $p$, of course this introduces some 'nonlocality', which is encoded in the frame itself.
\end{remark}

What also carries over from the classical setting directly is the naturality of the Levi-Civita connection: if $(M,\godd)$ and $(\tilde{M},\tilde{\godd})$ are two ODD Riemannian manifolds with associated Levi-Civita connections $\nabla$ and $\tilde{\nabla}$ respectively, and $\varphi \colon M \to \tilde{M}$ is an \emph{isometry} - which ultimately means that $\godd$ and $\varphi^*\tilde{\godd}$ agree on algebroid vector fields $X,Y \in \Alg(TM)$ -, then the pullback of $\tilde{\nabla}$ is equal to $\nabla$. The proof of~\cite[Prop.~5.13]{Lee18} applies verbatim.  

\begin{remark}
{\em
From the formulae of the \emph{ODD Christoffel symbols} $\Gamma^k_{ij}$, we see that the ODD Levi-Civita connection ``adds" zeros and poles only where $\godd$ degenerates, which one might already have expected.  This is the reason why we are able to define geodesics in a meaningful way, which we will do in the next subsection. However, we remark that the covariant derivatives of ODD vector fields may not be ODD vector fields anymore, as the following example shows.
}
\end{remark}

\begin{example}
 We consider $M:=\RR$ together with the ODD metric $\godd:= x^2\mathrm{dx}^{\odot 2}$ as in Example~\ref{ex:RegularCurveODDvsClass}. We have seen that $(E=\frac{\del_x}{|x|})$ is a global orthonormal frame and in particular equals $2  \dot{\gamma}$ for the ODD regular curve $\gamma$. In the standard coordinate frame $(\del_x)$, we compute
 $$
 \Gamma^1_{11}=\frac{1}{2x^2}\left( \frac{\del x^2}{\del x} + \frac{\del x^2}{\del x} - \frac{\del x^2}{\del x}\right)=\frac{1}{x}.
 $$
 So we have
\begin{align*}
    \nabla_{\del_x} \del_x &= \frac{1}{x} \del_x,
\\
 \nabla_{E} \del_x &= \frac{1}{|x|x} \del_x,
 \\
 \nabla_{\del_x} E &= \frac{1}{|x|} \nabla_{\del_x} \del_x + \frac{\del (1/|x|)}{\del x} \del_x = \frac{1}{|x|x} \del_x - \frac{1}{|x|x} \del_x =0,
 \\
 \nabla_{E} E &= \frac{1}{|x|} \nabla_{\del_x} E =0.
\end{align*}
In particular, $\nabla_{\del_x} \del_x$ and $\nabla_{E} \del_x$ are algebroid vector fields that are not ODD. 
\end{example}

\subsection{Covariant Derivatives along Curves}

First, we have to define the right notion of vector fields along ODD regular curves. We have already seen an account of this in Definition~\ref{def:ODD-regular-curve}.

\begin{definition}
\label{def:ODD-vf-along-curve}
Let $\gamma \colon I \to M$ be an ODD regular curve in an ODD Riemannian manifold $(M,\godd)$. Let $(E_i)_{1\leq i \leq n}$ be a local ODD orthonormal frame. We say that $E_i$ \emph{has no pole along $\gamma$}, if $\left.E_i\right|_{\gamma(I)}$ is well-defined for all but finitely many points of $\gamma(I)$. We denote 
$$
J:=\{1 \leq i \leq n ~|~ E_i \mathrm{~has~no~pole~along~} \gamma\}.
$$
An \emph{algebroid vector field along $\gamma$} is a map $V \colon I \to T_{\gamma(t)}M$ (defined on all but finitely many points of $I$), such that there are algebroid functions $V^i \in \Alg(I)$ satisfying
$$
V(t)=\sum_{i \in J} V^i(t) \left.E_i\right|_{\gamma(t)}.
$$
An algebroid vector field $V$ along $\gamma$ is called an \emph{ODD vector field along $\gamma$}, if the $V^i$ are analytic functions. We denote the space of ODD vector fields along $\gamma$ by $\mathfrak{X}(\gamma)$.
\end{definition}

\begin{remark}
{\em
We remark that if $\gamma$ lies in some degeneracy locus $N_i$, there may be classical vector fields along $\gamma$ orthogonal to $N_i$, that by the above definition are not ODD vector fields along $\gamma$ - because they do not lie in the span of $(E_j)_{j \in J}$. However, at least the tangent vector field $\dot{\gamma}$ is always an ODD vector field along $\gamma$.
}
\end{remark}

When $\gamma \colon I \to M$ is an ODD regular curve, then choosing ODD orthonormal frames $(E_i)_i$ instead of coordinate frames, we can show the existence of a covariant derivative $D_t$ along $\gamma$ along the same lines as in the classical case, cf.~\cite[Thm.~4.24]{Lee18}. Moreover, we see that the values of $\nabla_v Y$, where $v$ is some coefficient vector with respect to $(E_i)_i$ at $p \in M$, only \emph{depend on the values of $Y$ along any ODD regular curve with velocity $v$ at $p$}, compare~\cite[Prop.~4.26]{Lee18}.

\subsection{Integrability of ODD vector fields}
\label{sec:int-vf}
In this subsection, we aim to study how a meaningful integrability of ODD vector fields could look like. The first observation is that in general, \emph{integral curves of ODD vector fields are not unique}, as the following example shows.  

\begin{example}
\label{ex:nonuniqueintegralcurves}
We consider $M:=\RR^2$ with the ODD metric 
$$
\godd:=(x^2+y^2)(\mathrm{dx}^{\odot 2} + \mathrm{dy}^{\odot 2})+2(y^2-x^2)\mathrm{dx} \odot \mathrm{dy},
$$
which degenerates along the coordinate axes. An ODD orthonormal frame is given by
$$
E_1:=\frac{1}{\sqrt{x^2+y^2}} \del_x,
\qquad
E_2:=\frac{(x^2-y^2)\del_x+(x^2+y^2)\del_y}{2 \sqrt{x^2y^2(x^2+y^2)}}.
$$

First, we consider the ODD vector field $E_2$. 
The rescaled field 
$$
E_2':=2\sqrt{\frac{x^2y^2}{x^2+y^2}}E_2=\del_y+\frac{x^2-y^2}{x^2+y^2} \del_x
$$ 
is defined and nonvanishing away from the origin. In particular, it has well-defined integral curves through all points of the coordinate axes apart from the origin. Since $E_2'$ is a rescaling of $E_2$, reparametrizing these integral curves yields ODD integral curves of $E_2$. It remains the origin. 

We can check that there is in fact one integral line $y=ax$ of $E_2$ through the origin, since the equation $a^3+a^2-a+1$ has one real root. The picture below shows the (normalized) vector field $E_2'$ together with the integral line $y=ax$ through the origin in red.

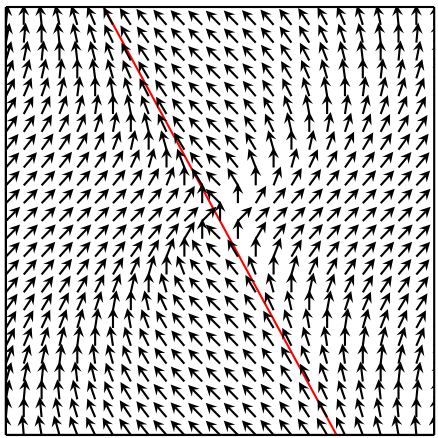
\begin{center}
\begin{tikzpicture}
\begin{axis}[
    xmin = -5, xmax = 5,
    ymin = -5, ymax = 5,
    zmin = 0, zmax = 1,
    axis equal image,
    view = {0}{90},
    ticks=none
]
\draw[color=red]   plot (\x,\x * -1.8393,0);

    \addplot3[
        quiver = {
            u = {(x^2-y^2)/(3*sqrt(x^4+y^4))},
            v = {(x^2+y^2)/(3*sqrt(x^4+y^4))},
        },
        -stealth,
        domain = -5:5,
        domain y = -5:5,
    ] {0};
\end{axis}
\end{tikzpicture}
\end{center}

It is easy to check that there are no other integral curves through the origin. So, in the case of $E_2$, we see that ODD integral curves exist and are unique (the same holds for $E_1$, with an easier argument). 
Now, consider the ODD vector field
$$
E_1+E_2= \frac{2\sqrt{x^2y^2} + x^2-y^2}{2 \sqrt{x^2y^2(x^2+y^2)}} \del_x + \frac{(x^2+y^2)}{2 \sqrt{x^2y^2(x^2+y^2)}} \del_y.
$$
The same scaling argument as before provides us with unique ODD integral curves through all points apart from the origin. However, now there are \emph{two} ODD integral lines $y=ax$ through the origin, with $a$ the two real solutions to
$$
2a\sqrt{a^2}+a-a^3-a^2-1=0.
$$
In the picture below, the (normalized) ODD vector field $E_1+E_2$ together with these two lines is depicted. Moreover, all ODD integral curves through points in the area bordered by these two lines go through the origin as well.

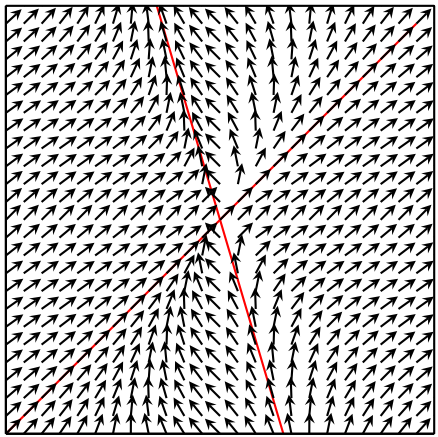
\begin{center}
\begin{tikzpicture}
\begin{axis}[
    xmin = -5, xmax = 5,
    ymin = -5, ymax = 5,
    zmin = 0, zmax = 1,
    axis equal image,
    view = {0}{90},
    ticks=none
]
\draw[color=red]   plot (\x,\x * -3.3830,0);
\draw[color=red]   plot (\x,\x ,0);

    \addplot3[
        quiver = {
            u = {(2*sqrt(x^2*y^2)+x^2-y^2)/(2*sqrt(2*x^4+2*y^4+4*x^2*y^2+4*x^2*sqrt(x^2*y^2)-4*y^2*sqrt(x^2*y^2)))},
            v = {(x^2+y^2)/(2*sqrt(2*x^4+2*y^4+4*x^2*y^2+4*x^2*sqrt(x^2*y^2)-4*y^2*sqrt(x^2*y^2)))},
        },
        -stealth,
        domain = -5:5,
        domain y = -5:5,
    ] {0};
\end{axis}
\end{tikzpicture}
\end{center}

Of course, in general the `border cases' will not always be lines.

\end{example}

However, as the following statement shows, integral curves exist and are unique at least at general points of maximal $N_j$ (as of course they do for general points of $M$).

\begin{theorem}[Integrability of ODD vector fields at general points of maximal components of $\mathcal{D}$]
\label{thm:integrability-odd-vf}
Let $(M,\godd)$ be an ODD Riemannian manifold and $p \in M$. Let $X$ be an ODD vector field. If either
\begin{enumerate}
    \item $p$ is a general point of $M$, or
    \item $p$ is a general point of some $N_j \subseteq \mathcal{D}$, which is maximal in the sense that it is not contained in any other $N_k$,
\end{enumerate}
then, there exists $ \epsilon>0$ and a unique ODD integral curve $\gamma\colon (-\epsilon,\epsilon) \to M$ with $p(0)=p$.
\end{theorem}

\begin{proof}
If $X$ has a zero at $p$, the constant curve $\gamma = p$ does the job, so we can assume that locally around $p$, $X$ has no zeros. Moreover, Item (1) follows from the classical Riemannian theory, since $\godd$ is Riemannian on a neighbourhood of a general point.

It remains to prove Item (2). We choose a coordinate neighbourhood $(x^i)_i$ centered at $p=0$ with associated coordinate frame $(\del_i)_i$, such that locally, $N_j$ and therefore $\mathcal{D}$ is given by the vanishing of the coordinates $x_{1},\ldots,x_{r}$ for some $r \leq n$.

We let $(E_i)_i$ be an ODD orthonormal frame defined on the same neighbourhood of $p$. Then 
$
X=X^i E_i
$
for some analytic functions $X^i$ by the definition of an ODD vector field. As in the remark above, we denote by $E_i^j$ the coefficient of $\del_j$ in the expansion of $E_i$ in the basis $(\del_j)_j$. In particular, $E_i^j$ is a continuous algebroid function, such that $(E_i^j)^2$ is meromorphic and has poles only at $N_j$. More concretely, when $(E_i)_i$ arises from $(\del_i)_i$ by the Gram-Schmidt orthogonalization process (which we will assume from now on), then we have
$$
E_i^j=\frac{(-1)^{\delta_{ij}+1}\frac{\godd_{ij}}{\godd_{jj}}}{\left| \del_i-\sum_{k=1}^{i-1} \frac{\godd_{ik}}{\godd_{kk}}\del_k\right|_\godd}.
$$

Thus, the existence of an integral curve through $p$ is equivalent to a local solution of the first order system
\begin{align*}
    \dot{x}^i(t) &= X^j(x(t)) E_j^i(x(t)), \\
     x^i(0) &= 0.
\end{align*}
We cannot use the Picard-Lindelöf theorem directly to obtain a local solution.
Instead, we multiply with an appropriate common denominator $h(x)$ of the $X^j(x(t)) E_j^i(x(t))$, which has zeros along $N_j$. Now we investigate the modified system 
\begin{align*}
    \dot{y}^i(t) &= X^j(y(t)) E_j^i(y(t)) h(y(t)), \\
     y^i(0) &= 0,
\end{align*}
which has a positive well-defined algebroid right hand side, which means the right hand side is analytic. So by the Picard-Lindelöf theorem, there is an analytic integral curve $y:I \to M$ with $y(0)=p$ solving the system. 

If $h(y(t))$ was constant zero on $I$, then $y$ would lie in $N_j$. But this means that $X$ would be tangent to $N_i$ and have a pole along $y(I)$. But then by the definition of an ODD metric, $\left.\godd\right|_{y(I)}(X,X)$ can not be an analytic function, which is a contradiction.

So we can assume $f(t):=h(y(t))$ is not constant equal to zero and has the only zero at zero (by choosing the interval of definition $I$  small enough). We consider the order one differential equation
\begin{align*}
    \frac{dk}{dt} &= \frac{1}{f(k)} \\
     k(0) &= 0.
\end{align*}
We see that we can solve the inverse system $\frac{dt}{dk}=f(k)$ by simple integration. Now the point is that we know that the first derivative $f(k)$ of the solution $t(k)$ is semipositive and has its only zero at zero. In particular, $t(k)$ is strictly monotonic and has a continuous inverse $k(t)$ around zero with an infinite slope $\frac{dk}{dt}(0)= \infty$ at zero. Now, we let $x(t):=y(k(t))$. With this definition, $x$ solves our initial first order system, as the following computation shows:
\begin{align*}
    \dot{x}^i(t) &=  \frac{d y^i(k(t))}{dt} = \dot{k}(t) \dot{y}^i(k(t)) = \frac{1}{f(k(t))} X^j(y(k(t))) E_j^i(y(k(t))) h(y(k(t))) = X^j(x(t)) E_j^i(x(t)),
    \\
    x(0) &= y(k(0)) = y(0) = 0.
\end{align*}
That is, $x$ is an ODD integral curve through $p$ with $\dot{x}(t)=\left. X \right|_{x(t)}$. 

\end{proof}

As already mentioned in the introduction, we believe that integral curves exist (though might be non-unique as Example~\ref{ex:nonuniqueintegralcurves} shows) also in more special points (i.e. special points that are not general points of maximal $N_j$). We also describe in the following conjecture how in a certain sense, still uniqueness holds and we at least give a sketch of how a possible proof should work.

\begin{conjecture}
\label{conj:int-vf}
Let $(M,\godd)$ be an ODD Riemannian manifold, $p \in M$, and $X$ an ODD vector field without zeros. There exists $ \epsilon>0$ and \emph{at least one} ODD regular integral curve $\gamma\colon (-\epsilon,\epsilon) \to M$ with $p(0)=p$. 

These curves are unique in the following sense:
If $\gamma$ and $\mu$ are two integral curves through $p$ and $\epsilon$ is chosen small enough such that $\gamma, \mu$ are regular curves away from $p$, then
$$
\lim_{t \to 0} \frac{\dot{\gamma}(t)}{|\dot{\gamma}(t)|_\godd} \neq \lim_{t \to 0} \frac{\dot{\mu}(t)}{|\dot{\mu}(t)|_\godd}.
$$
\end{conjecture}

One could paraphrase the last statement by saying that for every classical tangent direction at $p$, there is at most one integral curve of $X$ going through $p$ in that direction. Also note that the limites are well-defined since $\gamma$ and $\mu$ are \emph{regular} ODD integral curves (cf. Definition~\ref{def:ODD-regular-curve} (2)).

\begin{proof}[Sketch of Proof of Conjecture~\ref{conj:int-vf}]
As in the proof of Theorem~\ref{thm:integrability-odd-vf}, we can assume that the existence of an integral curve amounts to a solution of the first order system
\begin{align*}
    \dot{x}^i(t) &= X^j(x(t)) E_j^i(x(t)), \\
     x^i(0) &= 0.
\end{align*}
If now we multiply by the least common multiple $h$ of the denominators of the $X^j(x(t)) E_j^i(x(t))$, then the resulting modified system 
\begin{align*}
    \dot{y}^i(t) &= X^j(y(t)) E_j^i(y(t)) h(y(t)), \\
     y^i(0) &= 0,
\end{align*}
could have a zero at $p$. This means that integral curves w.r.t. this modified system that have $p$ in their closure would correspond to ODD integral curves for $X$ through $p$. Of course there can be more than one such curve or a priori none. We propose the following strategy:

\begin{itemize}
    \item To address the existence of curves with $p$ in their closure, we note that since $X$ has no zeros, the modified system \emph{will have an index zero singularity} at $p$. From this, one should be able to prove that due to continuity reasons there is at least one integral curve with $p$ in it's closure.  We also remark that by an appropriate induction on the codimension of $N_j$ where $p$ is a general point, one should be able to assume that away from $N_j$, locally we have integral curves through all points. 
    
    To be more precise, taking an appropriate transversal submanifold $N$ through $p$ and a small (maybe appropriately deformed) ball  $B_\epsilon(p)$, there are integral curves going through all points of $N \cap B_\epsilon(p)$ apart from $p$ in one direction entering $B_\epsilon(p)$ somewhere else. Due to continuity, the set of entry points of these curves must be of the same homotopy type, i.e. have a hole $S$ of codimension $n-1$. Curves entering `through that hole' will have $p$ in their closure. The picture below illustrates this approach.
    
    
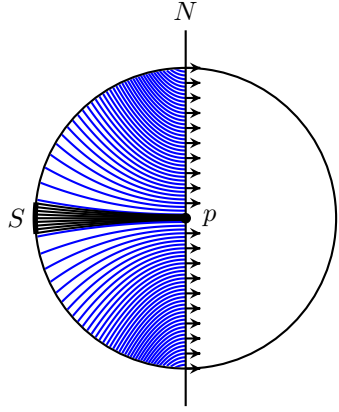
\begin{center}
    \begin{tikzpicture}[domain=-2:2]

\foreach \eps in {0.05, 0.1, ..., 2}
   \draw[ color=blue, domain=-2+\eps-0.1:0,  smooth,variable=\x]  plot(\x,\eps*\x*\x+\eps);
   \foreach \eps in {0, 0.01, ..., 0.04}
   \draw[color=black, domain=-2+\eps-0.1:0,  smooth,variable=\x]    plot(\x,\eps*\x*\x+\eps);
   \foreach \eps in {0.01, 0.02, ..., 0.05}
   \draw[color=black, domain=-2+\eps-0.1:0,  smooth,variable=\x]    plot(\x,-\eps*\x*\x-\eps);
   \foreach \eps in {0.05, 0.1, ..., 2}
   \draw[ color=blue, domain=-2+\eps-0.1:0,  smooth,variable=\x]    plot(\x,-\eps*\x*\x-\eps);

  \draw[fill=white,draw=white] (0,2.5) -- (0,2) arc (90:270:2) -- (0,-2.5)--(-2.2,-2.5)--(-2.2,2.5) -- cycle;
  
  \foreach \eps in {-2,-1.8,...,-0.2}
   \draw[ color=black,->,  domain=0:0.2,  smooth,variable=\x]    plot(\x,\eps);
   
   \foreach \eps in {0.2,0.4,...,2}
   \draw[ color=black,->,  domain=0:0.2,  smooth,variable=\x]    plot(\x,\eps);
  
  \coordinate[label=90:$N$] (X) at (0,2.5);
  \coordinate(P) at (0,0);
  \coordinate[label=0:$p$] (P') at (0.1,0);
\fill (P) circle (2pt);
\draw (0,-2.5) -- (0,2.5);
\draw (0,0) circle [radius=2cm];

\coordinate[label=-180:$S$] (S) at (-2,0);
\draw [line width=2pt,domain=174:186] plot ({2*cos(\x)}, {2*sin(\x)});
   
  \end{tikzpicture}
      \end{center}

    \item What the above approach can not prove is that: 
    \begin{enumerate}
        \item  reparametrizing these curves with $p$ in their closure yields \emph{ODD regular curves} $\gamma$, i.e. with 
    $$
        \lim_{t \nearrow 0} \frac{\dot{\gamma}(t)}{|\dot{\gamma}(t)|_\godd} = \lim_{t \searrow 0} \frac{\dot{\gamma}(t)}{|\dot{\gamma}(t)|_\godd},
    $$
    \item and  two \emph{different} such reparametrized curves $\gamma$ and $\mu$ through $p$ have \emph{different} tangential directions:
    $$
\lim_{t \to 0} \frac{\dot{\gamma}(t)}{|\dot{\gamma}(t)|_\godd} \neq \lim_{t \to 0} \frac{\dot{\mu}(t)}{|\dot{\mu}(t)|_\godd}.
$$
    \end{enumerate}
   
   To study this, we propose to investigate the behaviour of small perturbations
   \begin{align*}
    \dot{y}^i(t) &= X^j(y(t)) E_j^i(y(t)) h(y(t)) + \epsilon, \\
     y^i(0) &= 0,
\end{align*}
 for $\epsilon \searrow 0$. In particular, these perturbated systems have no singularities. Again, choose a small ball $B:=B_\epsilon(p)$ around $p$ and let $S$ be the (closed) set of points on the boundary $\del B$ such that for $s \in S$, $p$ lies in the closure of the integral curve (w.r.t. our initial system) through $s$. Now let again $N$ be an appropriate submanifold through $p$ transversal to the initial system. Denote $P_\epsilon$ the intersection of $N$ with integral curves of $\gamma$ through points in $S$ and $R_\epsilon$ the set of points where integral curves through points in $S$ leave $B$.
 
 For every $\epsilon>0$ there are bijections 
 $$
 \phi_\epsilon \colon S \to P_\epsilon, \quad
 \psi_\epsilon \colon S \to R_\epsilon,
 $$ 
 and for every $s \in S$, we have $\lim_{\epsilon \searrow 0} \phi_\epsilon (s) =p$. Now based on the fact that the singularity $p$ arises from a singularity free vector field in the ODD orthonormal frame, it should be possible to show that for $s\neq s' \in S$ and $\gamma_{s,\epsilon},\gamma_{s',\epsilon}$ integral curves of the system perturbed by $\epsilon$ starting at $s$ and $s'$ respectively, translated such that $\gamma_{s,\epsilon}(0)=\phi_\epsilon(s)$ and $\gamma_{s',\epsilon}(0)=\phi_\epsilon(s')$, we get
 $$
 \lim_{\epsilon \to 0} \frac{\dot{\gamma_s}(t)}{|\dot{\gamma_s}(t)|_\godd} \neq \lim_{\epsilon \to 0} \frac{\dot{\gamma_{s'}}(t)}{|\dot{\gamma_{s'}}(t)|_\godd},
 $$
 which would yield what we want.
 
 \end{itemize}

 \begin{center}
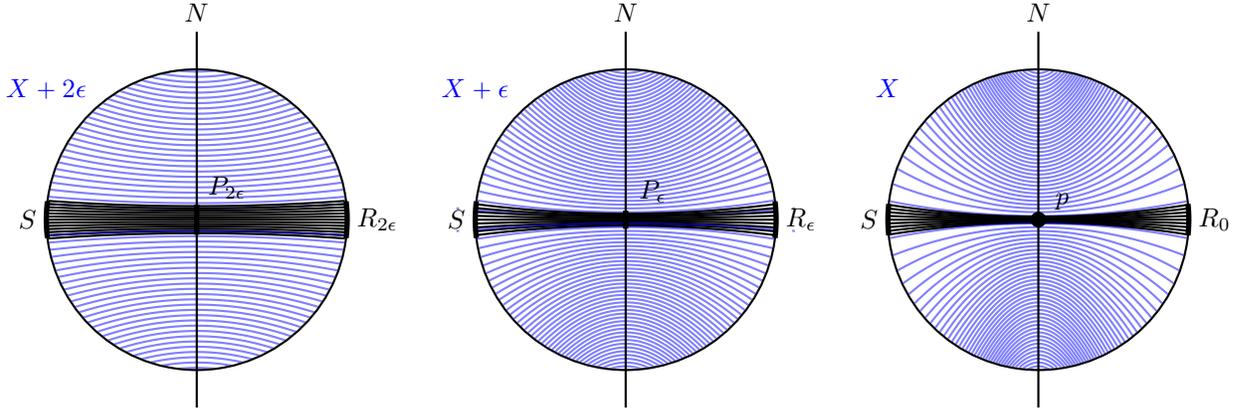
 \begin{tabular}{ccc}
 
\begin{tikzpicture}[domain=-2:2]
   
  \foreach \eps in {0.15, 0.22, ..., 2}
   \draw[opacity=0.5, color=blue, domain=-2:0,  smooth,variable=\x]  plot(\x,\eps*\x*\x*0.1+\eps);
   \foreach \eps in {0, 0.03, ..., 0.18}
   \draw[color=black, domain=-2:0,  smooth,variable=\x]    plot(\x,\eps*\x*\x*0.1+\eps);
   \foreach \eps in {0.03, 0.06, ..., 0.18}
   \draw[color=black, domain=-2:0,  smooth,variable=\x]    plot(\x,-\eps*\x*\x*0.1-\eps);
   \foreach \eps in {0.15, 0.22, ..., 2}
   \draw[opacity=0.5, color=blue, domain=-2:0,  smooth,variable=\x]    plot(\x,-\eps*\x*\x*0.1-\eps);
   
   \foreach \eps in {0.15, 0.22, ..., 2}
   \draw[opacity=0.5, color=blue, domain=-2:0,  smooth,variable=\x]  plot(-\x,\eps*\x*\x*0.1+\eps);
   \foreach \eps in {0, 0.03, ..., 0.18}
   \draw[color=black, domain=-2:0,  smooth,variable=\x]    plot(-\x,\eps*\x*\x*0.1+\eps);
   \foreach \eps in {0.03, 0.06, ..., 0.18}
   \draw[color=black, domain=-2:0,  smooth,variable=\x]    plot(-\x,-\eps*\x*\x*0.1-\eps);
   \foreach \eps in {0.15, 0.22, ..., 2}
   \draw[opacity=0.5, color=blue, domain=-2:0,  smooth,variable=\x]    plot(-\x,-\eps*\x*\x*0.1-\eps);
   
  \draw[fill=white,draw=white] (0,2.8) -- (0,2) arc (90:270:2) -- (0,-2.8)--(-2.2,-2.8)--(-2.2,2.8) -- cycle;
  \draw[fill=white,draw=white] (0,2.8) -- (0,2) arc (90:-90:2) -- (0,-2.8)--(2.2,-2.8)--(2.2,2.8) -- cycle;
  
    \coordinate[label=90:$N$] (N) at (0,2.5);
  \coordinate(P) at (0,0);
  \coordinate[label=80:$P_{2\epsilon}$] (P') at (0.02,0.15);
\draw [line width=2pt,domain=-0.2:0.2] plot (0,\x);
\draw (0,-2.5) -- (0,2.5);
\draw (0,0) circle [radius=2cm];

\coordinate[label=-180:$S$] (S) at (-2,0);
\draw [line width=2pt,domain=173:187] plot ({2*cos(\x)}, {2*sin(\x)});

\coordinate[label=0:$R_{2\epsilon}$] (R) at (2,0);
\draw [line width=2pt,domain=-7:7] plot ({2*cos(\x)}, {2*sin(\x)});

    \coordinate[color=blue,label=-90:\textcolor{blue}{$X+2\epsilon$}] (X) at (-2,2);
  \end{tikzpicture}
 
      & 
      
      \begin{tikzpicture}[domain=-2:2]

    \foreach \eps in {0.05, 0.1, ..., 2}
   \draw[opacity=0.5, color=blue, domain=-2+\eps-0.3:0,  smooth,variable=\x]  plot(\x,\eps*\x*\x*0.4+\eps);
   \foreach \eps in {0, 0.02, ..., 0.1}
   \draw[color=black, domain=-2+\eps-0.1:0,  smooth,variable=\x]    plot(\x,\eps*\x*\x*0.4+\eps);
   \foreach \eps in {0.02, 0.04, ..., 0.08}
   \draw[color=black, domain=-2+\eps-0.1:0,  smooth,variable=\x]    plot(\x,-\eps*\x*\x*0.4-\eps);
   \foreach \eps in {0.05, 0.1, ..., 2}
   \draw[opacity=0.5, color=blue, domain=-2+\eps-0.3:0,  smooth,variable=\x]    plot(\x,-\eps*\x*\x*0.4-\eps);
   
    \foreach \eps in {0.05, 0.1, ..., 2}
   \draw[opacity=0.5, color=blue, domain=-2+\eps-0.3:0,  smooth,variable=\x]  plot(-\x,\eps*\x*\x*0.4+\eps);
   \foreach \eps in {0, 0.02, ..., 0.1}
   \draw[color=black, domain=-2+\eps-0.1:0,  smooth,variable=\x]    plot(-\x,\eps*\x*\x*0.4+\eps);
   \foreach \eps in {0.02, 0.04, ..., 0.08}
   \draw[color=black, domain=-2+\eps-0.1:0,  smooth,variable=\x]    plot(-\x,-\eps*\x*\x*0.4-\eps);
   \foreach \eps in {0.05, 0.1, ..., 2}
   \draw[opacity=0.5, color=blue, domain=-2+\eps-0.3:0,  smooth,variable=\x]    plot(-\x,-\eps*\x*\x*0.4-\eps);
   
    \draw[fill=white,draw=white] (0,2.8) -- (0,2) arc (90:270:2) -- (0,-2.8)--(-2.2,-2.8)--(-2.2,2.8) -- cycle;
  \draw[fill=white,draw=white] (0,2.8) -- (0,2) arc (90:-90:2) -- (0,-2.8)--(2.2,-2.8)--(2.2,2.8) -- cycle;
  
    \coordinate[label=90:$N$] (N) at (0,2.5);
  \coordinate(P) at (0,0);
  \coordinate[label=80:$P_{\epsilon}$] (P') at (0.05,0.1);
\draw [line width=2pt,domain=-0.12:0.12] plot (0,\x);
\draw (0,-2.5) -- (0,2.5);
\draw (0,0) circle [radius=2cm];

\coordinate[label=-180:$S$] (S) at (-2,0);
\draw [line width=2pt,domain=173:187] plot ({2*cos(\x)}, {2*sin(\x)});

\coordinate[label=0:$R_{\epsilon}$] (R) at (2,0);
\draw [line width=2pt,domain=-7:7] plot ({2*cos(\x)}, {2*sin(\x)});

    \coordinate[color=blue,label=-90:\textcolor{blue}{$X+\epsilon$}] (X) at (-2,2);
  \end{tikzpicture}
      &
 \begin{tikzpicture}[domain=-2:2]
   
  \foreach \eps in {0.05, 0.1, ..., 2}
   \draw[opacity=0.5, color=blue, domain=-2+\eps-0.1:0,  smooth,variable=\x]  plot(\x,\eps*\x*\x+\eps);
   \foreach \eps in {0, 0.01, ..., 0.04}
   \draw[color=black, domain=-2+\eps-0.1:0,  smooth,variable=\x]    plot(\x,\eps*\x*\x+\eps);
   \foreach \eps in {0.01, 0.02, ..., 0.05}
   \draw[color=black, domain=-2+\eps-0.1:0,  smooth,variable=\x]    plot(\x,-\eps*\x*\x-\eps);
   \foreach \eps in {0.05, 0.1, ..., 2}
   \draw[opacity=0.5, color=blue, domain=-2+\eps-0.1:0,  smooth,variable=\x]    plot(\x,-\eps*\x*\x-\eps);
   
    \foreach \eps in {0.05, 0.1, ..., 2}
   \draw[opacity=0.5, color=blue, domain=-2+\eps-0.1:0,  smooth,variable=\x]  plot(-\x,\eps*\x*\x+\eps);
   \foreach \eps in {0, 0.01, ..., 0.04}
   \draw[color=black, domain=-2+\eps-0.1:0,  smooth,variable=\x]    plot(-\x,\eps*\x*\x+\eps);
   \foreach \eps in {0.01, 0.02, ..., 0.05}
   \draw[color=black, domain=-2+\eps-0.1:0,  smooth,variable=\x]    plot(-\x,-\eps*\x*\x-\eps);
   \foreach \eps in {0.05, 0.1, ..., 2}
   \draw[opacity=0.5, color=blue, domain=-2+\eps-0.1:0,  smooth,variable=\x]    plot(-\x,-\eps*\x*\x-\eps);
   
    \draw[fill=white,draw=white] (0,2.8) -- (0,2) arc (90:270:2) -- (0,-2.8)--(-2.2,-2.8)--(-2.2,2.8) -- cycle;
  \draw[fill=white,draw=white] (0,2.8) -- (0,2) arc (90:-90:2) -- (0,-2.8)--(2.2,-2.8)--(2.2,2.8) -- cycle;
  
    \coordinate[label=90:$N$] (X) at (0,2.5);
  \coordinate(P) at (0,0);
  \coordinate[label=80:$p$] (P') at (0.1,0);
\fill (P) circle (3pt);
\draw (0,-2.5) -- (0,2.5);
\draw (0,0) circle [radius=2cm];

\coordinate[label=-180:$S$] (S) at (-2,0);
\draw [line width=2pt,domain=174:186] plot ({2*cos(\x)}, {2*sin(\x)});

\coordinate[label=0:$R_0$] (R) at (2,0);
\draw [line width=2pt,domain=-6:6] plot ({2*cos(\x)}, {2*sin(\x)});

    \coordinate[color=blue,label=-90:\textcolor{blue}{$X$}] (X) at (-2,2);
  \end{tikzpicture}
 
 \end{tabular}

      \end{center}

\end{proof}

\begin{remark}
{\em
We expect that there will be an ODD version of the \emph{Fundamental Theorem on Flows}  (c.f.~\cite[Thm.~9.12]{Lee13}) and that on a compact ODD Riemannian manifold, every ODD vector field is complete, i.e. all maximal integral curves are defined on the whole of $\RR$ (cf.~\cite[Le.~9.15,~Thm.~9.16]{Lee13}).
}
\end{remark}

\begin{remark}
{\em
We also remark that when $\godd$ is flat in the sense that there is an orthonormal coordinate frame around every point, then Theorem~\ref{thm:integrability-odd-vf} and Conjecture~\ref{conj:int-vf} already yield the existence of ODD geodesics. 
}
\end{remark}

\subsection{ODD Geodesics}

By the existence of a covariant derivative $D_t$, at least we can define geodesics as in the classical case.

\begin{definition}
Let $\gamma \colon I \to M$ be an ODD regular curve. We say that $\gamma$ is an \emph{ODD geodesic}, if $D_t \dot{\gamma} \equiv 0$.
\end{definition}

The \emph{existence} of ODD geodesics, compared to the classical case (see, e.g.,~\cite[Thm.~4.27]{Lee18}), is a harder task, as in the case of integral curves of vector fields. First, for $p \in M$, we should ask ourselves what an appropriate ODD tangential direction for a geodesic would be? By the definition of ODD vector fields (Def.~\ref{def:ODD-vectorfield}) and ODD vector fields along curves (Def.~\ref{def:ODD-vf-along-curve}), it should a priori be defined as follows:

\begin{definition}
Let $p \in (M,\godd)$, an ODD Riemannian manifold. Let   $(E_i)_i$ be a local orthonormal frame around $p$. An \emph{ODD tangential direction at $p$} is a formal sum
$$
\vodd:=\vodd^i \left. E_i \right|_{p},
$$
with $\vodd^i \in \RR$ for $1\leq i \leq n$.
\end{definition}

\begin{remark}
When $p$ lies in the degeneracy locus of $\godd$, then an ODD tangential direction may \emph{not} be a classical tangent vector at $p$. However, we can always assume to be in the following setting: let $(x^i)$ be a coordinate neighbourhood of $p$ and $(\del_i)$ be the associated coordinate frame. Let $(E_i)$ be the orthonormal frame obtained from $(\del_i)$ by the Gram-Schmidt process. Then expanding $E_i$ in the $\del_i$ yields $E_i=E_i^j \del_j$ with algebroid functions $E_i^j$.
This means we can formally write
$$
\vodd=\vodd^i \left. E_i \right|_{p} =\vodd^i E_i^j(p) \left. \del_j \right|_{p}.
$$
However, as we have seen in Example~\ref{ex:nonuniqueintegralcurves}, an ODD tangential direction will not provide a unique classical tangential direction, so we cannot expect geodesics in special points to be unique.
\end{remark}

When we try to prove the existence of ODD geodesics with this definition and analogous to the classical case (c.f.~\cite[Thm.~4.27]{Lee18}, we encounter a second order system of differential equations not only with an algebroid right hand side, but also with \emph{algebroid initial conditions} given by $\mathfrak{v}$. It is not clear (at least to us) how to tackle such a system. 

However, since by definition ODD geodesics are ODD regular curves, we know they have ODD reparametrizations which are classical regular curves due to Proposition~\ref{prop:ODDcurvesareregular}. So modulo reparametrization, it suffices to find geodesics for classical tangential directions. As for vector fields, we only give the definitive statement for general points of maximal $N_j$, while we formulate it as a conjecture for more special points.

\begin{theorem}[Existence of ODD Geodesics at general points of maximal components of $\mathcal{D}$]
\label{thm:existence-geodesics}
Let $(M,\godd)$ be an ODD Riemannian manifold and $p \in M$. Let $v \in T_pM$ be a tangent vector at $p$. If either
\begin{enumerate}
    \item $p$ is a general point of $M$, or
    \item $p$ is a general point of some $N_j \subseteq \mathcal{D}$, which is maximal in the sense that it is not contained in any other $N_k$,
\end{enumerate}
then, there exists a unique ODD geodesic $\gamma$ through $p$ in direction $v$.
\end{theorem}

\begin{proof}
Let $(x^i)_i$ be a coordinate neighbourhood centered at $p$ with corresponding coordinate frame $(\del_i)_i$. We can assume that the ODD orthormal frame $(E_i)_i$ arises from $(\del_i)_i$ by the Gram-Schmidt process and that $\vodd= \vodd^i \left. E_i\right|_{p}$ for some $\vodd^i \in \RR$.

Let $\gamma$ be an ODD regular curve through $p$ with coefficient functions $x_i(t)$. We let $\Gamma_{ij}^k$ be the Christoffel symbols in the frame $(\del_i)_i$. Then we compute
\begin{align*}
    D_t \dot{\gamma}(t) &= \ddot{x}^i(t) \left.\del_i\right|_{\gamma(t)} + \dot{x}^j(t) \nabla_{\dot{\gamma}(t)} \left. \del_j \right|_{\gamma(t)} \\
    &= \ddot{x}^i(t) \left.\del_i\right|_{\gamma(t)} + \dot{x}^j(t)\dot{x}^k(t) \nabla_{\del_k} \left. \del_j \right|_{\gamma(t)} \\
    &= \ddot{x}^i(t) + \dot{x}^j(t)\dot{x}^k(t) \Gamma_{kj}^i(\gamma(t)) \left.\del_i\right|_{\gamma(t)}.
\end{align*}

So $\gamma$ is an ODD geodesic through $p$ in direction $v$ if and only if the \emph{geodesic equations} $\ddot{x}^i(t) + \dot{x}^j(t)\dot{x}^k(t) \Gamma_{kj}^i(\gamma(t))=0$ together with the initial conditions are satisfied for all $1 \leq i \leq n$. We convert this second order system to a first order system as usual by introducing new variables $v^i(t)$ and imposing
\begin{align*}
    \dot{x}^i(t) &= v^i(t), \\
    \dot{v}^i(t) &= - v^j(t)v^k(t) \Gamma_{kj}^i(x^1(t),\ldots,x^n(t)),
\end{align*}
together with the initial conditions
\begin{align*}
    x^i(0) &= 0, \\
    v^i(0) &= v^i.
\end{align*}
Now we can proceed as in the proof of Theorem~\ref{thm:integrability-odd-vf} by multiplying this system appropriately.
\end{proof}

The general case should read as follows:

\begin{conjecture}[Existence of ODD Geodesics]
\label{conj:existence-geodesics}
Let $(M,\godd)$ be an ODD Riemannian manifold and $p \in M$ with an ODD orthonormal frame $(E_i)_i$ around $p$. Let $v \in T_pM$ be a tangent vector at $p$. Then, there exists a unique ODD geodesic $\gamma$ through $p$ in direction $v$.
\end{conjecture}

We expect that this could be proven along the same lines as Conjecture~\ref{conj:int-vf}.

\bibliographystyle{habbrv}
\bibliography{bib}

\begin{thebibliography}{10}
\expandafter\ifx\csname url\endcsname\relax
  \def\url#1{\texttt{#1}}\fi
\expandafter\ifx\csname doi\endcsname\relax
  \def\doi#1{\burlalt{doi:#1}{http://dx.doi.org/#1}}\fi
\expandafter\ifx\csname urlprefix\endcsname\relax\def\urlprefix{URL }\fi
\expandafter\ifx\csname href\endcsname\relax
  \def\href#1#2{#2}\fi
\expandafter\ifx\csname burlalt\endcsname\relax
  \def\burlalt#1#2{\href{#2}{#1}}\fi

\bibitem{KE1}
X.~Chen, S.~Donaldson, and S.~Sun.
\newblock K\"{a}hler-{E}instein metrics on {F}ano manifolds. {I}:
  {A}pproximation of metrics with cone singularities.
\newblock {\em J. Amer. Math. Soc.}, 28(1):183--197, 2015.
\newblock \doi{10.1090/S0894-0347-2014-00799-2}.

\bibitem{KE2}
X.~Chen, S.~Donaldson, and S.~Sun.
\newblock K\"{a}hler-{E}instein metrics on {F}ano manifolds. {II}: {L}imits
  with cone angle less than {$2\pi$}.
\newblock {\em J. Amer. Math. Soc.}, 28(1):199--234, 2015.
\newblock \doi{10.1090/S0894-0347-2014-00800-6}.

\bibitem{KE3}
X.~Chen, S.~Donaldson, and S.~Sun.
\newblock K\"{a}hler-{E}instein metrics on {F}ano manifolds. {III}: {L}imits as
  cone angle approaches {$2\pi$} and completion of the main proof.
\newblock {\em J. Amer. Math. Soc.}, 28(1):235--278, 2015.
\newblock \doi{10.1090/S0894-0347-2014-00801-8}.

\bibitem{DK81}
D.~M. DeTurck and J.~L. Kazdan.
\newblock Some regularity theorems in {R}iemannian geometry.
\newblock {\em Ann. Sci. \'{E}cole Norm. Sup. (4)}, 14(3):249--260, 1981.
\newblock
  \urlprefix\url{http://www.numdam.org/item?id=ASENS_1981_4_14_3_249_0}.

\bibitem{Er82}
A.~E. Er\"{e}menko.
\newblock Meromorphic solutions of algebraic differential equations.
\newblock {\em Uspekhi Mat. Nauk}, 37(4(226)):53--82, 1982. English
  translation: Russian Math. Surveys 37 (1982), no. 4, 61–95.

\bibitem{EGZ09}
P.~Eyssidieux, V.~Guedj, and A.~Zeriahi.
\newblock Singular {K}\"{a}hler-{E}instein metrics.
\newblock {\em J. Amer. Math. Soc.}, 22(3):607--639, 2009.
\newblock \doi{10.1090/S0894-0347-09-00629-8}.

\bibitem{HTT08}
R.~Hotta, K.~Takeuchi, and T.~Tanisaki.
\newblock {\em {$D$}-modules, perverse sheaves, and representation theory},
  volume 236 of {\em Progress in Mathematics}.
\newblock Birkh\"{a}user Boston, Inc., Boston, MA, 2008.
\newblock \doi{10.1007/978-0-8176-4523-6}.
\newblock Translated from the 1995 Japanese edition by Takeuchi.

\bibitem{Lee13}
J.~M. Lee.
\newblock {\em Introduction to smooth manifolds}, volume 218 of {\em Graduate
  Texts in Mathematics}.
\newblock Springer, New York, second edition, 2013.

\bibitem{Lee18}
J.~M. Lee.
\newblock {\em Introduction to {R}iemannian manifolds}, volume 176 of {\em
  Graduate Texts in Mathematics}.
\newblock Springer, Cham, 2018.
\newblock Second edition.

\bibitem{LXZ22}
Y.~Liu, C.~Xu, and Z.~Zhuang.
\newblock Finite generation for valuations computing stability thresholds and
  applications to {K}-stability.
\newblock {\em Ann. of Math. (2)}, 196(2):507--566, 2022.
\newblock \doi{10.4007/annals.2022.196.2.2}.

\bibitem{RT11}
J.~Ross and R.~Thomas.
\newblock Weighted projective embeddings, stability of orbifolds, and constant
  scalar curvature {K}\"{a}hler metrics.
\newblock {\em J. Differential Geom.}, 88(1):109--159, 2011.
\newblock \urlprefix\url{http://projecteuclid.org/euclid.jdg/1317758871}.

\bibitem{Str86}
R.~S. Strichartz.
\newblock Sub-{R}iemannian geometry.
\newblock {\em J. Differential Geom.}, 24(2):221--263, 1986.
\newblock \urlprefix\url{http://projecteuclid.org/euclid.jdg/1214440436}.

\end{thebibliography}

\end{document}